\documentclass[pdflatex,sn-mathphys-num]{sn-jnl}

\usepackage{graphicx}%
\usepackage{multirow}%
\usepackage{amsmath,amssymb,amsfonts}%
\usepackage{amsthm}%
\usepackage{mathrsfs}%
\usepackage[title]{appendix}%
\usepackage{xcolor}%
\usepackage{textcomp}%
\usepackage{manyfoot}%
\usepackage{booktabs}%
\usepackage{algorithm}%
\usepackage{algorithmicx}%
\usepackage{algpseudocode}%
\usepackage{listings}%

\usepackage{geometry}
\usepackage{siunitx}
\usepackage{hyperref}

\begin{document}
	\newtheorem{Theorem}{Theorem}
    \newtheorem{Proposition}[Theorem]{Proposition}
    \newtheorem{Lemma}[Theorem]{Lemma}
	\newtheorem{Definition}[Theorem]{Definition}
    \newtheorem{Remark}[Theorem]{Remark}
    \newtheorem{Algorithm}[Theorem]{Algorithm}
    \newtheorem{Assumption}[Theorem]{Assumption}
    \newtheorem{Example}[Theorem]{Example}
    \numberwithin{equation}{section}
    \numberwithin{Theorem}{section}

\title[Error Analysis On EIs]{Error Analysis on a Novel Class of Exponential Integrators with Local Linear Extension Techniques for Highly Oscillatory ODEs}


\author[1]{\fnm{Zhihao} \sur{Qi}}\email{zhihaoqi@smail.nju.edu.cn}

\author*[1]{\fnm{Weibing} \sur{Deng}}\email{wbdeng@nju.edu.cn}

\author[1]{\fnm{Fuhai} \sur{Zhu}}\email{zhufuhai@nju.edu.cn}

\affil[1]{\orgdiv{School of Mathematics}, \orgname{Nanjing University}, \orgaddress{\city{Nanjing}, \postcode{210093}, \country{People's Republic of China}}}


\abstract{This paper investigates a class of non-autonomous highly oscillatory ordinary differential equations characterized by a linear component inversely proportional to a small parameter $\varepsilon$, with purely imaginary eigenvalues, and an $\varepsilon$-independent nonlinear part. When $0<\varepsilon\ll 1$, the rapidly oscillatory nature of the solution imposes severe constraints on step size selection and numerical accuracy, leading to considerable computational difficulties. Inspired by a linearization technique that introduces auxiliary polynomial variables, a new family of explicit exponential integrators has recently been proposed. These methods do not require the linear part to be diagonal or to have eigenvalues that are integer multiples of a fixed value—a common assumption in multiscale approaches—and they achieve arbitrarily high orders of convergence without imposing order conditions.
	The main contribution of this work is to provide a rigorous error analysis for this new class of methods under a bounded oscillatory energy condition. To this end, we first establish the equivalence between the high-dimensional system and the original problem using algebraic techniques. Building on these foundational results, we prove that the numerical schemes, when employing auxiliary polynomial variables of degree $k$, achieve a uniform convergence order of $O(h^{k+1})$. In particular, an improved order of $O(\varepsilon h^k)$ is attained {when $h$ is larger than the scale of $\varepsilon$.} These theoretical findings are further applied to second-order oscillatory systems, leading to improved uniform accuracy with respect to $\varepsilon$. Finally, numerical experiments confirm the optimality of the derived error estimates.}




\maketitle

\section{Introduction}\label{Sec1}
In this work, we study systems of ordinary differential equations (ODEs) of the form
\begin{equation}\label{P1}
	\begin{aligned}
		&\dot{\mathbf{u}}=\frac{1}{\varepsilon}A\mathbf{u}+\mathbf{F}(\mathbf{u},t), \quad t\in[0,T],\\
		&\mathbf{u}(0)=\varepsilon^\nu\mathbf{u}_{\mathrm{in}},
	\end{aligned}
\end{equation}
where $\mathbf{u}\in\mathbb{C}^d$ ($d\geq 1$) is the solution vector, $\varepsilon$ is a dimensionless parameter governing the oscillation frequency, $\nu\in\mathbb{R}$ characterizes the initial value scaling, $\mathbf{u}_{\mathrm{in}}\in\mathbb{C}^d$ is a given vector, and $T>0$ is an $\varepsilon$-independent final time. The matrix $A=(A_{ij})_{d\times d}$ is assumed to be diagonalizable with purely imaginary eigenvalues, though not necessarily diagonal. The nonlinear forcing term $\mathbf{F}$ depends explicitly on time and is Lipschitz continuous in all variables up to order {$k$}, with Lipschitz constants independent of $\varepsilon$.

For $0 < \varepsilon \ll 1$, system (\ref{P1}) describes a wide range of highly oscillatory problems dominated by linear oscillatory dynamics, as the $O(\varepsilon^{-1})$ scaling of the linear operator dictates the oscillatory behavior. The purely imaginary spectrum of $A$ implies that each solution component exhibits oscillations, behaving as a superposition of high-frequency waves with temporal wavelengths on the order of $O(\varepsilon)$. A representative example of such systems is the second-order oscillatory equation \cite{LJL2005, F2006, WYSW2010, CCMS2012, HLW2013, CLMZ2017}
\begin{equation}\label{OP}
\begin{aligned}
  &\ddot{\mathbf{y}}+\frac{1}{\varepsilon^2}M\mathbf{y}=\mathbf{g}(\mathbf{y},t),\quad  t\in [0,T],\\
  &\mathbf{y}(0)=\varepsilon^\nu\mathbf{y}_{\mathrm{in}},\quad \dot{\mathbf{y}}(0)=\varepsilon^{\nu-1}\dot{\mathbf{y}}_{\mathrm{in}},
\end{aligned}
\end{equation}
where $M$ is a symmetric positive-definite matrix and $\mathbf{g}$ denotes a nonlinear function. Such highly oscillatory systems are widespread in scientific computing and emerge in numerous disciplines, including classical mechanics, molecular dynamics, quantum physics, astronomy, circuit simulation, and plasma physics \cite{CJLL2006, GSS1998, HL2000, HLW2013, WYW2013}. We introduce the scaled momentum variable $\mathbf{p}=\varepsilon \dot{\mathbf{y}}$ and set $\mathbf{u}=[\mathbf{y}^\top,\mathbf{p}^\top]^\top$, thereby reformulating system (\ref{OP}) into the form (\ref{P1}).

If the right-hand side $\mathbf{g}$ of (\ref{OP}) is time-independent and there exists a potential function $U(\mathbf{y})$ satisfying $\mathbf{g}(\mathbf{y})=-\nabla U(\mathbf{y})$, then (\ref{OP}) represents a Hamiltonian system \cite{LR2004}. The Hamiltonian takes the form
\begin{equation}\label{Hamiltonian}
  H(\mathbf{y},\dot{\mathbf{y}})=H_1(\mathbf{y},\dot{\mathbf{y}})+U(\mathbf{y})
\end{equation}
with
\begin{equation}\label{OscEnergy}
  H_1(\mathbf{y},\dot{\mathbf{y}}) =\frac{1}{2}\|\dot{\mathbf{y}}\|^2 +\frac{1}{2\varepsilon^2}\mathbf{y}^\top M\mathbf{y},
\end{equation}
where $\|\cdot\|$ denotes the standard Euclidean norm. The quantity $H_1$ defined in (\ref{OscEnergy}) represents the total kinetic and potential energy associated with the linear oscillator component of (\ref{OP}).

Systems characterized by bounded energy have attracted considerable attention in both theoretical and applied research, particularly in classical mechanics, where such models are widely applicable \cite{Z2017, CJLL2006}. More generally, the right-hand side of (\ref{OP}) may model both internal forces and external time-dependent interactions, extending the applicability of (\ref{OP}) beyond conservative systems. Nevertheless, the definition of oscillatory energy (\ref{OscEnergy}) remains applicable. Since $H_1$ typically dominates the total energy $H$ in highly oscillatory problems, we therefore introduce the bounded oscillatory energy condition
\begin{equation}\label{boec}
  H_1(\mathbf{u})\leq C,\quad t\in[0,T],
\end{equation}
as a more general energy constraint for non-autonomous systems. Here and in the following, $C$ represents a generic constant independent of both $h$ and $\varepsilon$. Under this condition, $\nu\geq 1$ holds, which implies that the amplitude of high-frequency oscillators in (\ref{OP}) is comparable to $\varepsilon$ while their velocity remains $O(1)$ as the frequency increases.

In this paper, we focus on the more general first-order formulation (\ref{P1}) with the physically relevant bounded oscillatory energy (\ref{boec}) to perform a theoretical analysis of numerical schemes. Therefore, combining (\ref{boec}) with the definition of $H_1$ directly implies that, for problem (\ref{P1}), the bounded oscillatory energy condition is equivalent to
\begin{equation}\label{bdener}
\mathbf{\|u\|} \le C\varepsilon.
\end{equation}

From a computational perspective, the highly oscillatory nature of the solution presents fundamental challenges for numerical treatment, rendering many classical numerical schemes inefficient or even ineffective. The dominant linear terms in system~(\ref{P1}) impose severe stability constraints on conventional numerical methods, including standard Runge–Kutta and Runge–Kutta–Nyström schemes \cite{FQ2010,HLW2013}. To preserve numerical reliability, the time step size must resolve the fastest oscillations, necessitating the restrictive condition $h < C\varepsilon$ \cite{Z2017,BDZ2014}. Simulating such repetitive oscillatory motions incurs substantial practical costs and demands considerable computational resources, thereby severely limiting practical efficiency.

Furthermore, when applied to (\ref{P1}), most conventional numerical methods exhibit errors that grow with increasing system frequency. These methods are typically constructed based on Taylor expansions of the exact solution, whose remainder terms dominate the truncation error. As a result, the global error generally depends on negative powers of $\varepsilon$, since temporal derivatives of the solution scale as $1/\varepsilon^k$ for some $k > 0$ \cite{WJ2024,CG2021}. In other words, such traditional approaches lack uniform accuracy with respect to $\varepsilon$, rendering them unsuitable for solving (\ref{P1}) in the regime of very small $\varepsilon$.

To address these numerical challenges, Gautschi first proposed in the 1960s the use of trigonometric polynomials based on the variation-of-constants formula to handle high-frequency components \cite{G1961}. This inspired the subsequent development of various Gautschi-type trigonometric integrators, which incorporate different filter functions \cite{HL1999,G2005,GH2006}. A systematic investigation of the long-term behavior of numerical solutions has been carried out using the modulated Fourier expansion technique introduced in \cite{HL2000}, with further contributions provided in \cite{CHL2003,CHL2005,C2006,HLW2020,WZ2023}. When system (\ref{P1}) models highly oscillatory Hamiltonian dynamical systems \cite{LR2004}, extensive research has been devoted to designing structure-preserving numerical solvers (see \cite{HL2000,HL1999,GH2006,WWX2013,BDZ2014,WWM2017,CXWJ2021} and references therein). Such methods can yield reliable numerical approximations after substantial time steps by preserving the inherent geometric structure of the original oscillatory system in an equivalent long-time dynamical sense. Nevertheless, most of these approaches remain subject to strict stability constraints.

To overcome step size restrictions,  exponential integrators (EIs) have emerged as powerful tools for stiff differential equations. Over recent decades, EIs have attracted considerable attention and undergone substantial development. A number of influential studies have been conducted \cite{CM2002,HOS2009,K2005,T2006,CCO2008,MW2017}, with a comprehensive review available in \cite{HO2010}. The construction of EIs builds on the variation-of-constants formula, which enables exact treatment of the linear part of the system. Owing to their favorable stability properties \cite{CM2002}, such methods are particularly effective when stiffness is primarily confined to the linear component.

When applied to second-order systems of the form (\ref{OP}), the matrix exponential naturally takes the form of trigonometric functions \cite{WWX2013}. This connection has motivated the development of several specialized methods, including ARKN \cite{F2006}, ERKN \cite{WWX2013,WWM2017}, and Filon-type methods \cite{WLW2013,K2008}, many of which also possess excellent structure-preserving qualities.

Further advances have integrated EIs with linearization techniques \cite{DBCOJ2007,KO2013,RG1997,T2006}, as exemplified by Exponential Rosenbrock-type methods \cite{CO2009,HOS2009,LL2023}. These approaches often achieve enhanced accuracy and efficiency while simplifying order conditions. However, whether linearization is employed or not, conventional EIs are generally derived under the assumption that solutions to stiff systems are temporally smooth. This assumption is frequently violated in highly oscillatory problems, leading to a loss of uniform accuracy with respect to $\varepsilon$.

Over the past two decades, a variety of multiscale methods have been proposed that achieve both $\varepsilon$-independent step sizes and accuracy. The core strategy behind these approaches is to leverage high-frequency characteristics to reformulate the original system into an oscillation-free form, thereby enabling the design of numerical methods with uniform accuracy. Methods such as the Heterogeneous Multiscale Method \cite{E2003,AEEV2012,EELRV2007,CS2010,AEKLT2013}, the Multi-Revolution Composition Method \cite{CMMV2014,CCZ2018}, and the Stroboscopic Averaging Method \cite{CLMV2020,CCMS2012,CMMZ2016,CCMM2015a} focus primarily on capturing averaged behavior over multiple oscillation periods. Meanwhile, Multiscale Time Integrators \cite{BDZ2014,Z2017,BCJT2016} and Multiscale Exponential Wave Integrators \cite{BC2014,FXY2023,FY2021} are constructed based on asymptotic expansions of exact solutions. The Nested Picard Iteration scheme \cite{BFS2018,CG2021,CZ2022} applies the variation-of-constants formula iteratively, and the Two-Scale Formulation approach \cite{WZ2023,CLMZ2017,CCLM2015b} treats high-frequency components as additional variables in the discretization.
While these methods achieve uniform accuracy, many of them rely on structural assumptions regarding the matrix $A$. Typical requirements include that all eigenvalues be integer multiples of a fixed principal frequency, or that the matrix be diagonal—whether assumed directly or achieved through preprocessing steps such as diagonalization. These assumptions help decouple the oscillatory modes of the system, effectively reducing the dynamics to single-frequency behavior and thereby providing a foundational simplification for the construction of multiscale methods.

In this study, we provide a theoretical analysis of a novel class of EIs based on a local linear extension technique recently introduced in \cite{QD2025b}. Unlike the multiscale methods discussed earlier, which exploit inherent high-frequency features of the equations, the proposed approach relies primarily on a dimension-raising technique to achieve linearization. By operating within the EI framework, the resulting methods effectively utilize the variation-of-constants formula, which treats the linear part exactly without imposing special structural assumptions on the matrix $A$ or the nonlinear function $\mathbf{F}$. The local linearization strategy introduced here addresses a key challenge of achieving uniform accuracy with both small and large step sizes when applying conventional EIs to system (\ref{P1}) with bounded oscillatory energy. Moreover, these new EIs bypass order conditions entirely, thereby theoretically enabling the systematic construction of schemes with arbitrarily high orders of convergence.

The numerical solution of (\ref{P1}) proceeds according to the following computational framework. First, a set of auxiliary variables is defined as polynomial functions in $\mathbf{u}$ up to degree $k$ based on the numerical solution point $\mathbf{U}_n$ at each time step $t_n$. Governing equations for these auxiliary variables are then derived from the original system (\ref{P1}), leading to an equivalent high-dimensional ODE system referred to as the local linear extension system. Numerical approximations $\mathbf{U}_{n+1}$ to the exact solution $\mathbf{u}(t_{n+1})$ are obtained by applying EIs to this extended system, followed by a projection step that maps the high-dimensional solution back to the original state space.

A key insight underlying this approach is that, through the dimension-raising transformation, the $k$-th order Taylor expansion of the nonlinear term $\mathbf{F}$ becomes linear in the extension variable. Although the extension system retains a highly oscillatory character, its reconstructed linear part incorporates more information than that of the original system. The resulting EIs are fully explicit, ensuring computational efficiency while achieving uniform accuracy without the need for iterative procedures.

The main contribution of this work is to establish a rigorous theoretical analysis for the new class of EIs under the condition (\ref{bdener}). {In general, these methods achieve a uniform error bound of $O(h^{k+1})$ for the numerical solution, where $k$ is the highest degree of the auxiliary polynomial variables. Furthermore, the convergence order is improved to $O(\varepsilon h^k)$ for larger step sizes $h > c_0\varepsilon$, where $c_0$ is a constant independent of $h$ and $\varepsilon$.} The characteristic of high-order accuracy without step size restriction permits the methods to be particularly effective for problems in classical mechanics field. As a direct application of the general theory, we extend these results to the second-order system~(\ref{OP}), obtaining higher-order error estimates with improved uniform convergence.

The analysis builds fundamentally on the adiabatic transformation \cite{HO2010,JL2003,LJL2005} and periodic decomposition techniques, which allow each oscillatory component in the high-dimensional system to be treated individually. A particularly non-trivial and crucial observation is that the $O(1/\varepsilon)$-scaled linear part in the extension system preserves the spectral properties of the original matrix $A$. We rigorously prove this invariance via detailed algebraic derivations, a result that plays an essential role in establishing uniform accuracy. Another key challenge in the analysis arises from potential resonance phenomena among high-frequency components in the extension system, where standard periodic decomposition techniques are insufficient. To overcome this, we analyze the invariant subspaces and their matrix representations corresponding to the $O(1/\varepsilon)$-scaled high-dimensional operators. This approach successfully isolates and eliminates resonance effects from the error estimates.

The remainder of this paper is organized as follows. Section \ref{Sec2} introduces the construction of local linear extension exponential integrators. Section \ref{Sec3} analyzes the algebraic properties of the linear part in the local linear extension system, and Section~\ref{Sec4} establishes the convergence analysis under various parameter regimes. Several representative numerical experiments are presented in Section~\ref{Sec5} to validate the theoretical findings. Concluding remarks are provided in the final section.

\section{The local linear extension EIs}\label{Sec2}
This section concisely reviews the construction methodology for the local linear extension system and the associated EIs. We first treat time $t$ as an independent variable and define the augmented state vector $\mathbf{x}=[\mathbf{u}^\top,t]^\top$. The original system (\ref{P1}) is reformulated as
\begin{equation}\label{P}
\begin{aligned}
  &\dot{\mathbf{x}}=\frac{1}{\varepsilon}A_1\mathbf{x}+\mathbf{f}(\mathbf{x}),\\
  &\mathbf{x}(0)=\varepsilon^\nu\mathbf{x}_{\mathrm{in}},
\end{aligned}
\end{equation}
where
\begin{equation}\label{P02}
  A_1 = (A_1)_{ij}=\left[\begin{array}{cc}
              A & \mathbf{0}_{d\times 1} \\
              \mathbf{0}_{1\times d} & 0
            \end{array}\right],\quad
  \mathbf{f}(\mathbf{x})=\left[\begin{array}{c}
                                 \mathbf{F}(\mathbf{u},t) \\
                                 1
                               \end{array}\right],\quad
  \mathbf{x}_{\mathrm{in}}=\left[\begin{array}{c}
                       \mathbf{u}_{\mathrm{in}} \\
                       0
                     \end{array}\right].
\end{equation}
Here $\mathbf{0}_{d_1\times d_2}$ represents a $d_1\times d_2$ zero matrix.

Introducing local linear extension variables requires the preliminary definition of multi-indices and index sets with associated algebraic structures. For positive integers $k$ and $d$, we define the index set as $\mathcal{I}_{d+1}^{[[k]]} := \{1,\cdots,d+1\}^k$, the $k$-fold Cartesian product of the set $\{1,\cdots,d+1\}$. Each element $\alpha \in \mathcal{I}_{d+1}^{[[k]]}$ represents a multi-index of length $k$, denoted by $|\alpha| = k$. Conventionally, for $k=0$, we define $\mathcal{I}_{d+1}^{[[0]]}=\{\emptyset\}$ as the singleton set containing only the empty index with no components. We introduce the following equivalence relation on this index set $\mathcal{I}_{d+1}^{[[k]]}$.
\begin{Definition}\label{Idk}
(\romannumeral1) Let $S_k$ denote the symmetric group on $k$ elements. The right action of group $S_k$ on the index set $\mathcal{I}_{d+1}^{[[k]]}$ is defined as follows: for any permutation $\sigma\in S_k$ of $\{1,\cdots,k\}$ and multi-index $\alpha=(\alpha_1,\cdots,\alpha_k)\in\mathcal{I}_{d+1}^{[[k]]}$,
\begin{equation*}
  \sigma(\alpha):=(\alpha_{\sigma(1)},\cdots,\alpha_{\sigma(k)})\in\mathcal{I}_{d+1}^{[[k]]}.
\end{equation*}
(\romannumeral2) Define an equivalence relation $\sim$ on $\mathcal{I}_{d+1}^{[[k]]}$ induced by the group action: for any $\alpha,\beta\in\mathcal{I}_{d+1}^{[[k]]}$,
\begin{equation*}
  \alpha\sim\beta\iff\mathrm{~there~exists~}\sigma\in S_k\mathrm{~such~that~} \sigma(\alpha)=\beta.
\end{equation*}
The equivalence follows immediately from the group properties of $S_k$. We denote by $[\alpha]$ the equivalence class containing $\alpha$.
\end{Definition}
By arranging the components of each multi-index in ascending order, the subset of $\mathcal{I}_{d+1}^{[[k]]}$
\begin{equation}\label{LLEVIdk02}
  \tilde{\mathcal{I}}_{d+1}^{[[k]]}:=\{\alpha=(\alpha_1,\cdots,\alpha_k)\in\mathcal{I}_{d+1}^{[[k]]}:\alpha_1\leq\cdots\leq\alpha_k\}
\end{equation}
provides a natural choice of representative elements for each equivalence class. Although our subsequent construction of numerical schemes and analysis is independent of the choice of representatives, we fix the representatives to be those multi-indices having the form specified in (\ref{LLEVIdk02}) for convenience. For $\alpha\in\mathcal{I}_{d+1}^{[[k]]}$, we denote such representative multi-indices of $[\alpha]$ with form (\ref{LLEVIdk02}) by $\overline{\alpha}$.

Define $\mathcal{I}_{d+1}^{[k]}:=\bigcup_{j=0}^k \mathcal{I}_{d+1}^{[[j]]}$. For a given point $\hat{\mathbf{x}}=[\hat{\mathbf{u}}^\top,\hat{t}]^\top$ with $\hat{\mathbf{u}}\in\mathbb{C}^d$ and $\hat{t}\geq 0$, the polynomial induced by a multi-index $\alpha\in\mathcal{I}_{d+1}^{[k]}$ is given by
\begin{equation}\label{LLEVIdk01}
	(\mathbf{x}-\hat{\mathbf{x}})^\alpha:=\begin{cases}
		1, & \mbox{if } |\alpha|=0, \\
		(x_{\alpha_1}-\hat{x}_{\alpha_1})\cdots(x_{\alpha_j}-\hat{x}_{\alpha_j}), & \mbox{if } |\alpha|=j\geq 1,
	\end{cases}
\end{equation}
where $x_i$ and $\hat{x}_i$ represent the $i$-th component of $\mathbf{x}$ and $\hat{\mathbf{x}}$, respectively.
These preparations lead to the definition of local linear extension variables.

\begin{Definition}[Local linear extension variable]\label{LLEV}
(\romannumeral1) We first introduce the sets of polynomials in $\mathbf{x}$ of exact degree $j$ and up to degree $k$ (of the form (\ref{LLEVIdk01})):
\begin{equation}\label{LLEV02}
  P_{\mathbf{x}}^{[[j,\hat{\mathbf{x}}]]}=\left\{(\mathbf{x}-\hat{\mathbf{x}})^{\overline{\alpha}}:\overline{\alpha}\in\tilde{\mathcal{I}}_{d+1}^{[[j]]}\right\}, \quad
  P_{\mathbf{x}}^{[k,\hat{\mathbf{x}}]}=\bigcup_{j=0}^k P_{\mathbf{x}}^{[[j,\hat{\mathbf{x}}]]}.
\end{equation}
Letting $|\cdot|$ denote set cardinality, we set $D^{[k]} = \big|P_{\mathbf{x}}^{[k,\hat{\mathbf{x}}]}\big|$. We then define the $k$-th order local linear extension variable at $\hat{\mathbf{x}}$, denoted by $\mathbf{x}^{[k,\hat{\mathbf{x}}]}$, as a $D^{[k]}$-dimensional vector. Its first component is 1, components 2 through $(d+2)$ are given by $\mathbf{x}-\hat{\mathbf{x}}$, and the remaining components may follow any prescribed order.

(\romannumeral2) We introduce the $D^{[[j]]}$-dimensional vectors $\mathbf{x}^{[[j,\hat{\mathbf{x}}]]}$ for $j=0,\cdots,k$, where $D^{[[j]]}=\big|P_{\mathbf{x}}^{[[j,\hat{\mathbf{x}}]]}\big|$, representing vectors composed of $j$-th degree polynomial components uniquely from $P_{\mathbf{x}}^{[[j,\hat{\mathbf{x}}]]}$. We define the {projection matrix} that extracts components $m_1$ through $m_2$ from a $D^{[k]}$-dimensional vector as follows:
\begin{equation*}
  \Pi_{m_1}^{m_2}=[\mathbf{0}_{m\times (m_1-1)},I_m,\mathbf{0}_{m\times (D^{[k]}-m_2)}],\quad\textrm{with}\quad m=m_2-m_1+1,
\end{equation*}
where $I_m$ denotes the $m\times m$ identity matrix. Then we specify $\mathbf{x}^{[[0,\hat{\mathbf{x}}]]}=\Pi_1^1\mathbf{x}^{[k,\hat{\mathbf{x}}]}=1$ and $\mathbf{x}^{[[1,\hat{\mathbf{x}}]]}=\Pi_2^{d+2}\mathbf{x}^{[k,\hat{\mathbf{x}}]}=\mathbf{x}-\hat{\mathbf{x}}$.

(\romannumeral3) The partial differential operator induced by the multi-index $\alpha\in\mathcal{I}_{d+1}^{[k]}$ is defined as
\begin{equation*}
  \frac{\partial^{\alpha}}{\partial \mathbf{x}^\alpha}:=\begin{cases}
                    \operatorname{id}, & \mbox{if } |\alpha|=0, \\
                   \frac{\partial^j}{\partial x_{\alpha_1}\cdots \partial x_{\alpha_j}}, & \mbox{if } |\alpha|=j\geq 1.
                 \end{cases}
\end{equation*}
\end{Definition}
\begin{Remark}\label{Rmk0}
Any valid sorting scheme is permitted for the components excluding the first $(d+2)$-th of $\mathbf{x}^{[k,\hat{\mathbf{x}}]}$, as different choices lead to equivalent numerical schemes and theoretical analyses. For example, the study in \cite{QD2025b} chooses the sorting scheme
\begin{equation}\label{LLEV03}
\mathbf{x}^{[k,\hat{\mathbf{x}}]}=\left[\begin{array}{cccc}
  \mathbf{x}^{[[0,\hat{\mathbf{x}}]]}, & \left(\mathbf{x}^{[[1,\hat{\mathbf{x}}]]}\right)^\top, & \cdots, & \left(\mathbf{x}^{[[k,\hat{\mathbf{x}}]]}\right)^\top
\end{array}\right]^\top.
\end{equation}
The standard lexicographical order is applied to the multi-indices in $\mathcal{I}_{d+1}^{[[j]]}$ ($j=2,\cdots,k$), which uniquely specify a sorting for the components of $\mathbf{x}^{[[j,\hat{\mathbf{x}}]]}$ through (\ref{LLEVIdk01}).
\end{Remark}
By Definition \ref{LLEV}, the set $P_{\mathbf{x}}^{[k,\hat{\mathbf{x}}]}$ forms a basis spanning the space of polynomials of degree at most $k$ in variables $x_{1},\cdots,x_{d+1}$, denoted by $\mathbb{P}^k[\mathbf{x}]$. Employing the partial derivatives and polynomials in Definition \ref{LLEV}, we derive the $k$-th order Taylor expansion of $\mathbf{f}$ about a reference point $\hat{\mathbf{x}}$:
\begin{equation}\label{ParDiff01}
\mathbf{f}(\mathbf{x})=\sum_{j=0}^k\sum_{\overline{\alpha}\in\tilde{\mathcal{I}}_{d+1}^{[[j]]}} \frac{1}{\gamma(\overline{\alpha})} \frac{\partial^{\overline{\alpha}}\mathbf{f}(\hat{\mathbf{x}})}{\partial \mathbf{x}^{\overline{\alpha}}} (\mathbf{x}-\hat{\mathbf{x}})^{\overline{\alpha}} +\bar{\mathbf{r}}^{k+1}(\mathbf{x};\hat{\mathbf{x}}),
\end{equation}
where $\gamma(\alpha)=\prod\limits_{q=1}^{d+1}|\alpha_{\{q\}}|!$ with the set $\alpha_{\{q\}}=\{j:\alpha_j=q\}$ representing components of the multi-index $\alpha$ with their value $q$. The term $\bar{\mathbf{r}}^{k+1}(\mathbf{x};\hat{\mathbf{x}})$ denotes the $(k+1)$-th order Taylor remainder. According to (\ref{P02}), we have
\begin{equation}\label{ParDiff03}
  \bar{\mathbf{r}}^{k+1}(\mathbf{x};\hat{\mathbf{x}})
  =\left[\begin{array}{c}
      \mathbf{r}^{k+1}(\mathbf{u},t;\hat{\mathbf{u}},\hat{t})\\ 0
  \end{array}\right]
\end{equation}
with
\begin{equation}\label{ParDiff04}
	\begin{split}
  \mathbf{r}^{k+1}(\mathbf{u},t;\hat{\mathbf{u}},\hat{t}):=\mathbf{F}(\mathbf{u},t)-&\sum_{j=0}^k\sum_{\overline{\alpha}\in\tilde{\mathcal{I}}_{d+1}^{[[j]]}}\frac{1}{\gamma(\overline{\alpha})} \frac{\partial^{|\overline{\alpha}|}\mathbf{F}(\mathbf{u},t)}{\partial u_1^{|\overline{\alpha}_{\{1\}}|}\cdots\partial u_d^{|\overline{\alpha}_{\{d\}}|}\partial t^{|\overline{\alpha}_{\{d+1\}}|}}\\
  &\qquad(\mathbf{u}-\hat{\mathbf{u}})^{\chi(\overline{\alpha};d+1)}(t-\hat{t})^{|\overline{\alpha}_{\{d+1\}}|}{=O((\mathbf{x}-\hat{\mathbf{x}})^{k+1})}
  \end{split}
\end{equation}
denoting the $(k+1)$-th order Taylor remainder of $\mathbf{F}(\mathbf{u},t)$ expanded at $(\hat{\mathbf{u}},\hat{t})$. Here $\chi(\alpha;l)=(\alpha_1,\cdots,\alpha_{l-1},\alpha_{l+1},\cdots,\alpha_j)\in\mathcal{I}_{d+1}^{[[j-1]]}$ is the multi-index obtained by removing the $l$-th component from $\alpha\in\mathcal{I}_{d+1}^{[[j]]}$.

As the variables in $P_{\mathbf{x}}^{[k,\hat{\mathbf{x}}]}$ remain unknown, we now derive the governing equations for these auxiliary polynomial variables to close the system.
Consider a reference point $\hat{\mathbf{x}}$ and a multi-index $\alpha\in\mathcal{I}_{d+1}^{[k]}$ with its length $j$. Exploiting (\ref{P}) and (\ref{LLEVIdk01}), we obtain
\begin{align*}
\frac{\mathrm{d}}{\mathrm{d}t}(\mathbf{x}-\hat{\mathbf{x}})^\alpha
&=  \sum_{l=1}^{j} \frac{(x_{\alpha_1}-\hat{x}_{\alpha_1}) \cdots (x_{\alpha_j}-\hat{x}_{\alpha_j})}{x_{\alpha_l}-\hat{x}_{\alpha_l}} \frac{\mathrm{d}(x_{\alpha_l}-\hat{x}_{\alpha_l})}{\mathrm{d}t} \notag\\
&=  \sum_{l=1}^{j} (\mathbf{x}-\hat{\mathbf{x}})^{\chi(\alpha;l)} \left(\frac{1}{\varepsilon}\sum_{m=1}^{d+1}(A_1)_{\alpha_lm}x_m+f_{\alpha_l}(\mathbf{x})\right).
\end{align*}
Hence,
\begin{equation}\label{ParDiff02}
	\begin{split}
\frac{\mathrm{d}}{\mathrm{d}t}(\mathbf{x}-\hat{\mathbf{x}})^\alpha&=  \frac{1}{\varepsilon}\sum_{l=1}^{j}\sum_{m=1}^{d+1} (A_1)_{\alpha_lm}(\mathbf{x}-\hat{\mathbf{x}})^{\chi(\alpha;l)}(x_m-\hat{x}_m)\\
& \quad+\frac{1}{\varepsilon}\sum_{l=1}^{j}\sum_{m=1}^{d+1} (A_1)_{\alpha_lm}\hat{x}_m(\mathbf{x}-\hat{\mathbf{x}})^{\chi(\alpha;l)}\\
& \quad+\sum_{l=1}^{j}\sum_{m=0}^{k-j+1}\sum_{\overline{\beta}\in\tilde{\mathcal{I}}_{d+1}^{[[m]]}} \frac{1}{\gamma(\overline{\beta})} \frac{\partial^{\overline{\beta}}f_{\alpha_l}(\hat{\mathbf{x}})}{\partial \mathbf{x}^{\overline{\beta}}} (\mathbf{x}-\hat{\mathbf{x}})^{\chi(\alpha;l)}(\mathbf{x}-\hat{\mathbf{x}})^{\overline{\beta}} \\
& \quad+ \sum_{l=1}^{j} (\mathbf{x}-\hat{\mathbf{x}})^{\chi(\alpha;l)} \bar{r}_{\alpha_l}^{k-j+2}(\mathbf{x};\hat{\mathbf{x}}),
\end{split}
\end{equation}
where it employs a $(k-j+1)$-th order Taylor expansion of $f_{\alpha_l}$, with $f_{\alpha_l}$ and $\bar{r}_{\alpha_l}^{i}$ representing the ${\alpha_l}$-th components of the vector-valued functions $\mathbf{f}$ and $\bar{\mathbf{r}}^{i}$, respectively.

A key observation reveals that all terms in (\ref{ParDiff02}) have the expression as linear combinations of polynomials from $P_{\mathbf{x}}^{[k,\hat{\mathbf{x}}]}$, except those involving the nonlinear remainders $\bar{r}_{\alpha_l}^{k-j+2}(\mathbf{x};\hat{\mathbf{x}})$. We sum the coefficients to consolidate identical polynomials, organised according to a specific ordering scheme, enabling representation of all terms in (\ref{ParDiff02}) (excluding the remainder term) as inner products of coefficient vectors and $\mathbf{x}^{[k,\hat{\mathbf{x}}]}$. Through differentiation of each polynomial component in $\mathbf{x}^{[k,\hat{\mathbf{x}}]}$ by (\ref{ParDiff02}), we aggregate the resulting coefficient vectors and isolate $O({1}/{\varepsilon})$ terms from $O(1)$ components by the distinct scale. This process yields the coefficient matrices $A_1^{[k]}(\hat{\mathbf{x}})$ and $A_0^{[k]}(\hat{\mathbf{x}})$, thereby deriving the governing equation for the local linear extension variables $\mathbf{x}^{[k,\hat{\mathbf{x}}]}$:
\begin{equation}\label{LLES01}
  \frac{\mathrm{d}\mathbf{x}^{[k,\hat{\mathbf{x}}]}}{\mathrm{d}t} =\frac{1}{\varepsilon}A_1^{[k]}(\hat{\mathbf{x}})\mathbf{x}^{[k,\hat{\mathbf{x}}]} +A_0^{[k]}(\hat{\mathbf{x}})\mathbf{x}^{[k,\hat{\mathbf{x}}]} +\mathbf{R}^{[k]}(\mathbf{x}^{[k,\hat{\mathbf{x}}]};\hat{\mathbf{x}}).
\end{equation}
If the $i$-th row of (\ref{LLES01}) corresponds to the equation generated by the multi-index $\alpha$ of length $j$, then, by (\ref{ParDiff02}) and noting that the $(d+1)$-th row in (\ref{ParDiff03}) is zero, the $i$-th component of $\mathbf{R}^{[k]}(\mathbf{x}^{[k,\hat{\mathbf{x}}]};\hat{\mathbf{x}})$ is given by
\begin{equation}\label{LLES02}
  R_i^{[k]}(\mathbf{x}^{[k,\hat{\mathbf{x}}]};\hat{\mathbf{x}})=\sum_{l=1}^{j}\mathrm{I}_{\{\alpha_l\neq d+1\}} (\mathbf{x}-\hat{\mathbf{x}})^{\chi(\alpha;l)} r_{\alpha_l}^{k-j+2}(\mathbf{u},t;\hat{\mathbf{u}},\hat{t}),
\end{equation}
where the indicator function is defined as
\begin{equation*}
  \mathrm{I}_{\{\alpha_l\neq d+1\}}=
  \begin{cases}
    1, & \mbox{if~} \alpha_l\neq d+1, \\
    0, & \mbox{if~} \alpha_l= d+1.
  \end{cases}
\end{equation*}

Noting that there are no simple closed-form expressions for the matrices $A_1^{[k]}$ and $A_0^{[k]}$ in the compact formulation, we give a feasible algorithm to construct these matrices as follows. The primary methodology involves an element-wise process that assigns specific values to their respective positions. Specifically, we first get all multi-indices in $P_{\mathbf{x}}^{[k,\hat{\mathbf{x}}]}$ by traversing the row index of the matrices. When treating the equation induced by a certain multi-index, we locate the column indices for the associated polynomials on the right-hand side of (\ref{ParDiff02}), followed by assigning the coefficients in (\ref{ParDiff02}) to this specific position. We refer to Appendix \ref{apdx} for an algorithmic implementation and \cite{QD2025b} for a more detailed discussion.

Then we have the following definition of local linear extension system.
\begin{Definition}[Local linear extension]\label{LLES}
For a given reference point $\hat{\mathbf{x}}$ and positive integer $k$, let $\mathbf{x}^{[k,\hat{\mathbf{x}}]}$ denote the $k$-th order local linear extension variable at $\hat{\mathbf{x}}$. We define (\ref{LLES01}) as the $k$-th order local linear extension system of (\ref{P}) at $\hat{\mathbf{x}}$, and refer to the above procedure to generate system (\ref{LLES01}) as the local linear extension operations.
\end{Definition}
When constructing the local linear extension system, our approach employs the Taylor expansion of $\mathbf{f}(\mathbf{x})$ with respect to $\mathbf{x}$ in (\ref{ParDiff02}), rather than expanding $\mathbf{f}(\mathbf{x}(t))$ solely in $t$ where conventional EIs are conceptually based. This critical distinction ensures the linear component of (\ref{LLES01}) incorporates enhanced information than~(\ref{P}), while the resulting formulation maintains $\varepsilon$-independence, which is essential for designing uniformly accurate methods.

To obtain a numerical solution for system (\ref{P}), we discretize the time interval $[0,T]$ by employing uniform steps, denoted as $0=t_0<t_1<\cdots< t_N=T$, with a constant step size $h=t_n-t_{n-1}$ ($n=1,\cdots,N$). We now develop the numerical method for solving (\ref{P}) based on the local linear extension system (\ref{LLES01}).

Let $\mathbf{X}_n$ represent the numerical solution to $\mathbf{x}(t_n)$. Taking $\hat{\mathbf{x}} = \mathbf{X}_n$ as the reference point in Definition \ref{LLEV}, we construct the $k$-th order local linear extension variable $\mathbf{X}^{[k,\mathbf{X}_n]}$ associated with the numerical solution. As indicated in Definition \ref{LLEV}, evaluation at time $t_n$ yields $\mathbf{X}^{[k,\mathbf{X}_n]}(t_n)=e_1$, where $e_j$ denotes the $j$-th canonical basis vector. We define the initial value problem satisfied by the extension variable $\mathbf{X}^{[k,\mathbf{X}_n]}(t)$ over the interval $[t_n,t_{n+1}]$ as
\begin{equation}\label{LLES-num-truncation}
\begin{aligned}
  &\frac{\mathrm{d}\mathbf{X}^{[k,\mathbf{X}_n]}}{\mathrm{d}t} =\frac{1}{\varepsilon} A_1^{[k]}(\mathbf{X}_n) \mathbf{X}^{[k,\mathbf{X}_n]} +A_0^{[k]}(\mathbf{X}_n)\mathbf{X}^{[k,\mathbf{X}_n]}, \quad t_n\leq t\leq t_{n+1},\\
  &\mathbf{X}^{[k,\mathbf{X}_n]}(t_n)=e_1,
\end{aligned}
\end{equation}
which represents a linear truncation of (\ref{LLES01}), obtained by setting the reference point $\hat{\mathbf{x}}=\mathbf{X}_n $ and neglecting the remainder term $\mathbf{R}^{[k]}$. Its exact solution provides $\mathbf{X}^{[k,\mathbf{X}_n]}(t)$ at the next time step
\begin{equation}\label{LLEEI01}
  \mathbf{X}^{[k,\mathbf{X}_n]}(t_{n+1})=\mathrm{e}^{\left(\frac{1}{\varepsilon}A_1^{[k]}(\mathbf{X}_n) +A_0^{[k]}(\mathbf{X}_n)\right)h} e_1.
\end{equation}
Following Definition \ref{LLEV} and (\ref{LLEVIdk01}), we recover the numerical solution of (\ref{P}) at $t_{n+1}$ through the projection
\begin{equation}\label{LLEEI02}
  \mathbf{X}_{n+1}:=\Pi_2^{d+2}\mathbf{X}^{[k,\mathbf{X}_n]}(t_{n+1})+\mathbf{X}_n.
\end{equation}
The equivalence between (\ref{P}) and (\ref{P1}) leads to the numerical solution for $\mathbf{u}(t_{n+1})$ as
\begin{equation}\label{LLEEI03}
  \mathbf{U}_{n+1}:=\left[
  \begin{array}{cc}
  \mathbf{0}_{d\times 1} & I_d
  \end{array}\right]\mathbf{X}_{n+1}.
\end{equation}

The resulting scheme (\ref{LLEEI01})--(\ref{LLEEI03}) is fully explicit and requires no iterative procedures. The analysis in Section \ref{Sec4} demonstrates that these integrators achieve a convergence order of $k+1$ with respect to $h$ when employing $k$-th order extension variables, implying that the improved convergence order stems directly from augmenting the dimensionality of the matrices $A_1^ {[k]}$ and $A_0^{[k]}$. That is to say, these EIs can theoretically achieve an arbitrary high-order accuracy without requiring order conditions.

\section{Algebraic properties of $A_1^{[k]}$}\label{Sec3}
We adopt the following algebraic convention. For $0\leq i<j\leq k$, the linear span of the polynomial set $\bigcup\limits_{l=i+1}^jP_{\mathbf{x}}^{[[l,\hat{\mathbf{x}}]]}$ is naturally linearly isomorphic to the quotient space $\mathbb{P}^j[\mathbf{x}]/\mathbb{P}^i[\mathbf{x}]$. We identify these isomorphic spaces and, in subsequent analysis, simply denote this polynomial space by $\mathbb{P}^j[\mathbf{x}]/\mathbb{P}^i[\mathbf{x}]$.

\subsection{Linear extension lemma}\label{SSec3-2}
When implementing EIs, the performance of the numerical scheme is fundamentally determined by the properties of the linear part in (\ref{LLES-num-truncation}), which is primarily governed by $A_1^{[k]}(\hat{\mathbf{x}})$. Although all eigenvalues of $A_1$ have zero real parts, it is essential to verify whether $A_1^{[k]}(\hat{\mathbf{x}})$ generated via local linear extension operations preserves this property. The emergence of spurious eigenvalues with positive real parts—which are not inherent to the original system—would significantly complicate the dynamics of the high-dimensional extended system. Fortunately, the following proposition, as a key result of this subsection, rules out this possibility.

\begin{Proposition}[Linear extension lemma]\label{Lem3-5}
	Suppose the local linear extension operation applied to system (\ref{P}) yields $A_1^{[k]}(\hat{\mathbf{x}})$, where $A_1$ is defined in (\ref{P02}). Then $A_1^{[k]}(\hat{\mathbf{x}})$ is diagonalizable, and all its eigenvalues have zero real parts.
\end{Proposition}
A proof of this result is provided in the remainder of this subsection. Since the matrix $A_1^{[k]}(\hat{\mathbf{x}})$ depends intrinsically only on the linear $O(1/\varepsilon)$ term in system (\ref{P}), we define the operator $D_1$ as
\begin{equation}\label{A0}
	D_1(\mathbf{x})=\lim_{\varepsilon\to 0^+}\varepsilon\frac{\mathrm{d}\mathbf{x}(t)}{\mathrm{d}t}.
\end{equation}
Applying $D_1$ to systems (\ref{P}) and (\ref{LLES01}), respectively, gives
\begin{align}
	&D_1(\mathbf{x})=A_1\mathbf{x},\label{A1}\\
	&D_1(\mathbf{x}^{[k,\hat{\mathbf{x}}]})=A_1^{[k]}(\hat{\mathbf{x}})\mathbf{x}^{[k,\hat{\mathbf{x}}]}. \label{A2}
\end{align}
To proceed, we first characterize the structure of $A_1^{[k]}(\hat{\mathbf{x}})$ via the following two lemmas.

\begin{Lemma}\label{LemAA1}
Suppose the ordering scheme for the extension variables is defined by~(\ref{LLEV03}). Then the following properties of $A_1^{[k]}(\hat{\mathbf{x}})$ hold:

(\romannumeral1) All matrices $A_1^{[k]}(\hat{\mathbf{x}})$ generated with different reference points $\hat{\mathbf{x}}$ are similar.

(\romannumeral2) There exists an invertible lower triangular matrix $S^{[k]}(\hat{\mathbf{x}})$, such that
\begin{equation}\label{LemAA1-03}
  A_1^{[k]}(\mathbf{0}_{(d+1)\times 1})=(S^{[k]}(\hat{\mathbf{x}}))^{-1}A_1^{[k]}(\hat{\mathbf{x}})S^{[k]}(\hat{\mathbf{x}}).
\end{equation}

(\romannumeral 3) Suppose $S^{[k]}(\hat{\mathbf{x}})$ and its inverse $(S^{[k]}(\hat{\mathbf{x}}))^{-1}$ are partitioned in block form as
\begin{equation}\label{LemAA1-04}
\begin{aligned}
  S^{[k]}(\hat{\mathbf{x}})=&
  \left[\begin{array}{cccc}
          S_{0,0}(\hat{\mathbf{x}}) &  &  & \\
          S_{1,0}(\hat{\mathbf{x}}) & S_{1,1}(\hat{\mathbf{x}}) &  & \\
          \vdots &  & \ddots & \\
          S_{k,0}(\hat{\mathbf{x}}) & S_{k,1}(\hat{\mathbf{x}}) & \ldots & S_{k,k}(\hat{\mathbf{x}})
        \end{array}\right],\\
  (S^{[k]}(\hat{\mathbf{x}}))^{-1}=&
  \left[\begin{array}{cccc}
          \bar{S}_{0,0}(\hat{\mathbf{x}}) &  &  & \\
          \bar{S}_{1,0}(\hat{\mathbf{x}}) & \bar{S}_{1,1}(\hat{\mathbf{x}}) &  & \\
          \vdots &  & \ddots & \\
          \bar{S}_{k,0}(\hat{\mathbf{x}}) & \bar{S}_{k,1}(\hat{\mathbf{x}}) & \ldots & \bar{S}_{k,k}(\hat{\mathbf{x}})
        \end{array}\right].
\end{aligned}
\end{equation}
Then each subblock satisfies $S_{j,i},\bar{S}_{j,i}\in(\mathbb{P}^{j-i}[\hat{\mathbf{x}}]/\mathbb{P}^{j-i-1}[\hat{\mathbf{x}}])^{D^{[[j]]}\times D^{[[i]]}}$; that is, the entries of $S_{j,i}(\hat{\mathbf{x}})$ and $\bar{S}_{j,i}(\hat{\mathbf{x}})$ are homogeneous polynomials of degree $j-i$ in the variable $\hat{\mathbf{x}}$. In particular, we have
\begin{equation}\label{LemAA1-01}
  S_{j,j}(\hat{\mathbf{x}})=\bar{S}_{j,j}(\hat{\mathbf{x}})=I_{D^{[[j]]}},\quad S_{1,0}=-\hat{\mathbf{x}},\quad \bar{S}_{1,0}=\hat{\mathbf{x}}.
\end{equation}
\end{Lemma}

\begin{proof}
\textbf{Proof of (\romannumeral1):} We note that the components in local linear extension variables with different choices of reference point $\hat{\mathbf{x}}$ in Definition \ref{LLEV} span the same polynomial space. Therefore, the matrices $A_1^{[k]}(\hat{\mathbf{x}})$ derived from the action of $D_1$ on $P_\mathbf{x}^{[k,\hat{\mathbf{x}}]}$ with different $\hat{\mathbf{x}}$ exhibit the similarity relation.

\textbf{Proof of (\romannumeral2):} It suffices to consider the transition matrix from the basis $P_\mathbf{x}^{[k,\mathbf{0}]}$ to $P_\mathbf{x}^{[k,\hat{\mathbf{x}}]}$ and denote it by $S^{[k]}(\hat{\mathbf{x}})$. For any degree-$j$ polynomial in $P_\mathbf{x}^{[k,\hat{\mathbf{x}}]}$ of the form (\ref{LLEVIdk01}), its expansion in monomials of degree at most $j$ from $P_\mathbf{x}^{[k,\mathbf{0}]}$ shows that, under the ordering scheme (\ref{LLEV03}), the matrix $S^{[k]}(\hat{\mathbf{x}})$ has a lower-triangular structure.

\textbf{Proof of (\romannumeral3):} From the identity $\mathbf{x}^{[[1,\hat{\mathbf{x}}]]}=\mathbf{x}^{[[1,\mathbf{0}]]}-\hat{\mathbf{x}}\cdot 1$, we immediately obtain $\bar{S}_{1,0}=-S_{1,0}=\hat{\mathbf{x}}$. Now take a polynomial of the form (\ref{LLEVIdk01}) with a multi-index $|\alpha|=j\geq 1$. The coefficient of the degree-$i$ monomial in its expansion is a homogeneous polynomial of degree $j-i$ in $\hat{\mathbf{x}}$. Consequently, $S_{j,i}\in(\mathbb{P}^{j-i}[\hat{\mathbf{x}}]/\mathbb{P}^{j-i-1}[\hat{\mathbf{x}}])^{D^{[[j]]}\times D^{[[i]]}}$. In particular, since the leading term has coefficient $1$, we have $S_{j,j}(\hat{\mathbf{x}})=I_{D^{[[j]]}}$. Conversely, a degree-$j$ monomial in $P_\mathbf{x}^{[k,\mathbf{0}]}$ can be written as
\begin{equation*}
  \mathbf{x}^\alpha=((x_{\alpha_1}-\hat{x}_{\alpha_1})+\hat{x}_{\alpha_1})\cdots((x_{\alpha_j}-\hat{x}_{\alpha_j})+\hat{x}_{\alpha_j}),
\end{equation*}
and expanded as a linear combination of polynomials in $P_\mathbf{x}^{[k,\hat{\mathbf{x}}]}$. The same reasoning yields $\bar{S}_{j,i}\in(\mathbb{P}^{j-i}[\hat{\mathbf{x}}]/\mathbb{P}^{j-i-1}[\hat{\mathbf{x}}])^{D^{[[j]]}\times D^{[[i]]}}$ and $\bar{S}_{j,j}(\hat{\mathbf{x}})=I_{D^{[[j]]}}$.
\end{proof}

\begin{Lemma}\label{LemA2}
Suppose the ordering scheme for the extension variables is defined by~(\ref{LLEV03}). Then the matrix $A_1^{[k]}(\hat{\mathbf{x}})$ has the block bidiagonal structure
\begin{equation}\label{A4}
A_1^{[k]}(\hat{\mathbf{x}})=\left[\begin{array}{cccc}
            0 &  &  &   \\
            B_1^{[[1]]}(\hat{\mathbf{x}}) & A_1^{[[1]]} &  &  \\
             & \ddots & \ddots &   \\
             &   & B_1^{[[k]]}(\hat{\mathbf{x}}) & A_1^{[[k]]}
          \end{array}\right],
\end{equation}
where $A_1^{[[j]]}$ and $B_1^{[[j]]}(\hat{\mathbf{x}})(j=1,\cdots,k)$ are matrices of size $D^{[[j]]}\times D^{[[j]]}$ and $D^{[[j]]}\times D^{[[j-1]]}$, respectively. In particular, $A_1^{[[1]]}=A_1$.
\end{Lemma}
\begin{proof}
For any multi-index $\alpha\in\mathcal{I}_{d+1}^{[[k]]}(k\geq 1)$, we have
\begin{align*}
D_1((\mathbf{x}-\hat{\mathbf{x}})^\alpha)=&\sum_{j=1}^{k}(x_{\alpha_1}-\hat{x}_{\alpha_1}) \cdots D_1(x_{\alpha_j}-\hat{x}_{\alpha_j})\cdots(x_{\alpha_k}-\hat{x}_{\alpha_k})\\
=&\sum_{j=1}^{k}(x_{\alpha_1}-\hat{x}_{\alpha_1}) \cdots\left(\sum_{l=1}^{d+1}(A_1)_{\alpha_jl}x_{l}\right)\cdots(x_{\alpha_k}-\hat{x}_{\alpha_k}).
\end{align*}
Hence
\begin{equation}\label{A3}
\begin{split}
D_1((\mathbf{x}-\hat{\mathbf{x}})^\alpha)=&\sum_{j=1}^{k}\sum_{l=1}^{d+1}(A_1)_{\alpha_jl}(x_{\alpha_1}-\hat{x}_{\alpha_1})\cdots(x_{l}-\hat{x}_{l})\cdots(x_{\alpha_k}-\hat{x}_{\alpha_k})\\
&\quad +\sum_{j=1}^{k}\sum_{l=1}^{d+1}(A_1)_{\alpha_jl}\hat{x}_{l} \frac{(x_{\alpha_1}-\hat{x}_{\alpha_1})\cdots(x_{\alpha_k}-\hat{x}_{\alpha_k})}{x_{\alpha_j}-\hat{x}_{\alpha_j}}.
\end{split}
\end{equation}
The first double summation in (\ref{A3}) corresponds to a $k$-th degree polynomial, while the second one represents a $(k-1)$-th degree polynomial. This fact determines the unique matrices $A_1^{[[k]]}\in\mathbb{C}^{D^{[[k]]}\times D^{[[k]]}}$ and $B_1^{[[k]]}(\hat{\mathbf{x}})\in\mathbb{C}^{D^{[[k]]}\times D^{[[k-1]]}}$ satisfying
\begin{equation*}
D_1(\mathbf{x}^{[[k,\hat{\mathbf{x}}]]})=A_1^{[[k]]}\mathbf{x}^{[[k,\hat{\mathbf{x}}]]}+B_1^{[[k]]}(\hat{\mathbf{x}})\mathbf{x}^{[[k-1,\hat{\mathbf{x}}]]}.
\end{equation*}
In the case of $k = 0$, the corresponding row in $A_1^{[k]}(\hat{\mathbf{x}})$ is entirely zero. Therefore, $A_1^{[k]}(\hat{\mathbf{x}})$ possesses the structure (\ref{A4}).
\end{proof}

Without loss of generality, we set $\hat{\mathbf{x}} = \textbf{0}_{(d+1) \times 1}$ in the subsequent derivations. This choice is justified by conclusion (\romannumeral1) of Lemma \ref{LemAA1}, which guarantees that $A_1^{[k]}(\hat{\mathbf{x}})$ and $A_1^{[k]}(\textbf{0})$ share identical spectral properties and diagonalizability. Under this setting, it follows from (\ref{A3}) that all $B_1^{[[j]]}(\textbf{0})$ for $j = 1, \cdots, k$ reduce to zero matrices. As a result, expression (\ref{A4}) simplifies into the block-diagonal form
\begin{equation}\label{LemAA1-02}
	A_1^{[k]}(\textbf{0}) = \operatorname{diag}\Big\{0, A_1^{[[1]]}, \cdots, A_1^{[[k]]}\Big\},
\end{equation}
and the system of linear equations for $\mathbf{x}^{[[j,\mathbf{0}]]}$ becomes
\begin{equation}\label{A5}
	D_1\big(\mathbf{x}^{[[j,\mathbf{0}]]}\big) = A_1^{[[j]]} \mathbf{x}^{[[j,\mathbf{0}]]}, \quad j = 1, \cdots, k.
\end{equation}

To analyze the eigenvalues and diagonalizability of $A_1^{[k]}(\mathbf{0})$, it suffices to examine these properties for each diagonal block $A_1^{[[j]]}$ with $j=1,\cdots,k$. However, conventional analytical approaches—such as those based on characteristic polynomials and eigenspace decomposition—are not applicable here, due to the absence of explicit closed-form expressions for the matrices $A_1^{[[j]]}$.

To overcome this difficulty, we extend the action of the operator $D_1$ to a higher-dimensional space endowed with a tensor product structure. This allows us to investigate the relationship between a specific matrix representation of the extended operator and the matrices $A_1^{[[j]]}$. The following discussion illustrates this procedure for the case of $A_1^{[[k]]}$; the reasoning applies analogously to all other blocks $A_1^{[[j]]}$ with $j=1,\cdots,k-1$.

{Let $V=\mathbb{P}^1[\mathbf{x}]/\mathbb{P}^0[\mathbf{x}]$ with the basis $P_{\mathbf{x}}^{[[1,\mathbf{0}]]}$. Equation (\ref{A1}) defines a linear mapping on $V$ induced by $D_1$, whose matrix representation with respect to $P_{\mathbf{x}}^{[[1,\mathbf{0}]]}$ is $A_1$.} Finally, we denote by $V^k=V^{\otimes k}$ the $k$-fold tensor product space of $V$.

\begin{Lemma}\label{LemA3}
The operator $D_1$ can be extended to a linear mapping on the tensor product space $V^k$. Its matrix representation with respect to the tensor product basis $\{x_{\alpha_1}\otimes\cdots\otimes x_{\alpha_k}:\alpha\in\mathcal{I}_{d+1}^{[[k]]}\}$ is given by
\begin{equation}\label{A6}
  \mathcal{A}^{[[k]]}=\sum_{j=1}^{k}I_{d+1}^{\otimes (j-1)}\otimes A_1 \otimes I_{d+1}^{\otimes (k-j)},\quad j=1,\cdots,k,
\end{equation}
where the superscript $~^{\otimes i}$ denotes the $i$-fold tensor product.
\end{Lemma}

\begin{proof}
For any tensor $\eta = \eta_1 \otimes \cdots \otimes \eta_k \in V^k$, the definition of $D_1$ implies
\begin{equation*}
  D_1(\eta_1\cdot\eta_2\cdots \eta_k)=\sum_{j=1}^{k}\eta_1\cdots \eta_{j-1}\cdot D_1(\eta_j)\cdot\eta_{j+1}\cdots \eta_k.
\end{equation*}
This motivates the natural extension of $D_1$ to a mapping $D_1: V^k \to V^k$ defined by
\begin{equation}\label{A7}
	D_1(\eta_1 \otimes \cdots \otimes \eta_k) = \sum_{j=1}^{k} \eta_1 \otimes \cdots \otimes \eta_{j-1} \otimes D_1(\eta_j) \otimes \eta_{j+1} \otimes \cdots \otimes \eta_k,
\end{equation}
which is clearly linear. Now, taking $\eta$ to be a basis element of the form $x_{\alpha_1} \otimes \cdots \otimes x_{\alpha_k}$ with $\alpha \in \mathcal{I}_{d+1}^{[[k]]}$, and applying equations (\ref{A1}) and (\ref{A7}), we obtain
\begin{equation*}
  D_1(x_{\alpha_1}\otimes\cdots\otimes x_{\alpha_k})=\sum_{j=1}^{k}\sum_{l=1}^{d+1} (A_1)_{\alpha_jl}x_{\alpha_1}\otimes\cdots\otimes x_{\alpha_{j-1}}\otimes x_l\otimes x_{\alpha_{j+1}}\otimes\cdots\otimes x_{\alpha_k}.
\end{equation*}
This expression directly yields the matrix representation given in (\ref{A6}).
\end{proof}

From the diagonalizability of $A_1$ and the tensor product structure expressed in~(\ref{A6}), it follows directly that $\mathcal{A}^{[[k]]}$ is also diagonalizable. Moreover, formula (\ref{A6}) implies the following conclusion concerning its eigenvalues (see also \cite{H2008}).

\begin{Lemma}\label{LemA4}
	Let $\lambda(A_1)=\{\lambda_1,\cdots,\lambda_{d+1}\}\subset\mathbb{C}$ denote the spectrum of $A_1$. Then the spectrum of $\mathcal{A}^{[[k]]}$ is given by
	$$
	\lambda(\mathcal{A}^{[[k]]}) = \{\lambda_{j_1}+\cdots+\lambda_{j_k} \mid \lambda_{j_1},\cdots,\lambda_{j_k}\in\lambda(A_1)\}.$$
Consequently, all eigenvalues of $\mathcal{A}^{[[k]]}$ have zero real part.
\end{Lemma}

The preceding two lemmas establish the desired spectral and diagonalizability properties for the matrix $\mathcal{A}^{[[k]]}$, as stated in Proposition~\ref{Lem3-5}. We now turn to the relationship between $\mathcal{A}^{[[k]]}$ and our target matrix $A_1^{[[k]]}$. These two matrices arise from representing the operator $D_1$ on different spaces: $\mathcal{A}^{[[k]]}$ acts on the tensor product space $V^k$, while $A_1^{[[k]]}$ acts on the quotient space $\mathbb{P}^k[\mathbf{x}]/\mathbb{P}^{k-1}[\mathbf{x}]$. These spaces are not isomorphic. For example, the distinct tensors $x_1 \otimes x_2$ and $x_2 \otimes x_1$ in $V^2$ both correspond to the same polynomial $x_1x_2$ in $\mathbb{P}^2[\mathbf{x}]/\mathbb{P}^1[\mathbf{x}]$. However, their associated multi-indices belong to the same equivalence class defined in Definition~\ref{Idk}. Building on this connection, we now prove the following structural lemma.

\begin{Lemma}\label{LemA6}
	There exists a basis of $V^k$ with respect to which the matrix representation of $D_1$ is a block upper triangular matrix of the form
	\begin{equation}\label{A9}
		\begin{bmatrix}
			A_1^{[[k]]} & * \\
			 & *
		\end{bmatrix}.
	\end{equation}
\end{Lemma}

\begin{proof}
For any multi-index $\alpha\in\mathcal{I}_{d+1}^{[[k]]}$, positive integers $j\in\{1,\cdots,k\}$, and $l\in\{1,\cdots,d+1\}$, define $\iota(\alpha;j,l):=(\alpha_1,\cdots,\alpha_{j-1},l,\alpha_{j+1},\cdots,\alpha_k)\in\mathcal{I} _{d+1}^{[[k]]}$ as the multi-index obtained by replacing the $j$-th entry of $\alpha$ with $l$. Let $\bar{\iota}(\alpha;j,l)\in\tilde{\mathcal{I}}_{d+1}^{[[k]]}$ denote the representative of the equivalence class containing $\iota(\alpha;j,l)$, and denote its $i$-th component by $\bar{\iota}(\alpha;j,l)_i$.

Considering the polynomials generated by multi-indices in $\tilde{\mathcal{I}}_{d+1}^{[[k]]}$ on both sides of (\ref{A3}), we simplify this expression to
\begin{equation}\label{A10}
  D_1(\mathbf{x}^{\overline{\alpha}})=\sum_{j=1}^{k}\sum_{l=1}^{d+1}(A_1)_{\bar{\alpha}_jl} x_{\bar{\iota}(\overline{\alpha};j,l)_1}\cdots x_{\bar{\iota}(\overline{\alpha};j,l)_k}
  =\sum_{j=1}^{k}\sum_{l=1}^{d+1}(A_1)_{\bar{\alpha}_jl}\mathbf{x}^{\bar{\iota}(\overline{\alpha};j,l)}.
\end{equation}
This equation provides a linear representation of $D_1$ with respect to the basis $P_{\mathbf{x}}^{[[k,\mathbf{0}]]}$ and it coincides precisely with the row of (\ref{A2}) corresponding to $\mathbf{x}^{\overline{\alpha}}$. {We now construct a basis for $V^k$ under which the representation of $D_1$ takes the form (\ref{A9}).}

Any tensor $\eta\in V^k$ can be written as
\begin{equation*}
  \eta=\sum_{\alpha\in\mathcal{I}_{d+1}^{[[k]]}}c_\alpha x_{\alpha_1}\otimes\cdots\otimes x_{\alpha_k},\quad c_\alpha\in\mathbb{C}.
\end{equation*}
Define a linear mapping $\Phi:V^k\to V^k$ by
\begin{equation}\label{A11}
  \Phi(\eta)=\sum_{\overline{\alpha}\in\tilde{\mathcal{I}}_{d+1}^{[[k]]}}\sum_{\beta\in[\overline{\alpha}]}c_\beta x_{\bar{\alpha}_1} \otimes\cdots\otimes x_{\bar{\alpha}_k} + \sum_{\overline{\alpha}\in\tilde{\mathcal{I}}_{d+1}^{[[k]]}}\sum_{\beta\in[\overline{\alpha}]\setminus\{\overline{\alpha}\}}c_\beta x_{\beta_1} \otimes\cdots\otimes x_{\beta_k}.
\end{equation}
Its inverse is given explicitly by
\begin{equation}\label{A12}
  \Phi^{-1}(\eta)=\sum_{\overline{\alpha}\in\tilde{\mathcal{I}}_{d+1}^{[[k]]}}\left(c_{\overline{\alpha}}-\sum_{\beta\in[\overline{\alpha}]\setminus\{\overline{\alpha}\}}c_\beta\right)x_{\bar{\alpha}_1} \otimes\cdots\otimes x_{\bar{\alpha}_k} + \sum_{\overline{\alpha}\in\tilde{\mathcal{I}}_{d+1}^{[[k]]}}\sum_{\beta\in[\overline{\alpha}]\setminus\{\overline{\alpha}\}}c_\beta x_{\beta_1} \otimes\cdots\otimes x_{\beta_k}.
\end{equation}
The invertibility demonstrates that the set $\{\Phi(x_{\alpha_1}\otimes\cdots\otimes x_{\alpha_k}):\alpha\in\mathcal{I}_{d+1}^{[[k]]}\}$ forms a basis for $V^k$. {We define a subspace by
\begin{equation*}
\bar{V}^k := \operatorname{span}\{x_{\bar{\alpha}_1} \otimes \cdots \otimes x_{\bar{\alpha}_k} : \overline{\alpha} \in \tilde{\mathcal{I}}_{d+1}^{[[k]]}\} \subset V^k,
\end{equation*}
which is naturally isomorphic to the quotient space $\mathbb{P}^k[\mathbf{x}]/\mathbb{P}^{k-1}[\mathbf{x}]$. The equation (\ref{A11}) implies $\bar{V}^k$ is invariant under $\Phi$ (and $\Phi^{-1}$). In the following, we prove that $\{\Phi(x_{\alpha_1}\otimes\cdots\otimes x_{\alpha_k}):\alpha\in\mathcal{I}_{d+1}^{[[k]]}\}$ is the desired basis. To verify this fact,} we compute the action of $\Phi D_1\Phi^{-1}$ on the tensor induced by a given multi-index in $\mathcal{I}_{d+1}^{[[k]]}$, depending on whether the multi-index is a representative element.

\textbf{Case 1.} For each $\overline{\alpha}\in\tilde{\mathcal{I}}_{d+1}^{[[k]]}$, applying (\ref{A7}) and (\ref{A12}) gives
\begin{align*}
\Phi D_1\Phi^{-1}\left(x_ {\bar{\alpha}_1}\otimes\cdots\otimes x_{\bar{\alpha}_k}\right)=&\sum_{j=1}^k\Phi\left(x_ {\bar{\alpha}_1}\otimes\cdots\otimes D_1(x_ {\bar{\alpha}_j})\otimes\cdots\otimes x_{\bar{\alpha}_k}\right)\notag\\
=&\sum_{j=1}^k\sum_{l=1}^{d+1}(A_1)_{\bar{\alpha}_jl}\Phi(x_{\bar{\alpha}_1}\otimes\cdots\otimes x_l\otimes\cdots\otimes x_{\bar{\alpha}_k}).\notag
\end{align*}
From the definition of $\tilde{\mathcal{I}}_{d+1}^{[[k]]}$ in (\ref{LLEVIdk02}), we have $(\bar{\alpha}_1,\cdots,\bar{\alpha}_{j-1},l,\bar{\alpha}_{j+1},\cdots,\bar{\alpha}_k)\in\tilde{\mathcal{I}}_{d+1}^{[[k]]}$ if and only if $\bar{\alpha}_{j-1}\leq l\leq \bar{\alpha}_{j+1}$. Together with (\ref{A11}), this yields
\begin{equation}\label{A13}
\begin{split}
\Phi D_1\Phi^{-1}\left(x_ {\bar{\alpha}_1}\otimes\cdots\otimes x_{\bar{\alpha}_k}\right)=&\sum_{j=1}^k \bigg(\sum_{l=1}^{d+1}(A_1)_{\bar{\alpha}_jl}x_{\bar{\iota}(\overline{\alpha};j,l)_1}\otimes\cdots\otimes x_{\bar{\iota}(\overline{\alpha};j,l)_k} \\
&\quad+ \sum_{\substack{l<\bar{\alpha}_{j-1}\\l>\bar{\alpha}_{j+1}}}(A_1)_{\bar{\alpha}_jl} x_{\bar{\alpha}_1}\otimes\cdots\otimes x_l\otimes\cdots\otimes x_{\bar{\alpha}_k}\bigg).
\end{split}
\end{equation}
The first sum lies in $\bar{V}^k$ and coincides with the expression in (\ref{A10}), i.e., the coefficients correspond to a row of $A_1^{[[k]]}$. The second sum belongs to the complementary subspace $(\bar{V}^k)^\perp$. Consequently, with respect to the basis
$\{\Phi(x_{\alpha_1}\otimes\cdots\otimes x_{\alpha_k}):\alpha\in\mathcal{I}_{d+1}^{[[k]]}\}$,
the action of $D_1$ on $\Phi(x_ {\bar{\alpha}_1}\otimes\cdots\otimes x_{\bar{\alpha}_k})$ with $\overline{\alpha}\in\tilde{\mathcal{I}}_{d+1}^{[[k]]}$ is represented by the first block row of the matrix in (\ref{A9}).

\textbf{Case 2.} For any $\beta\in[\overline{\alpha}]\setminus\{\overline{\alpha}\}$, applying equations (\ref{A7}), (\ref{A11}) and (\ref{A12}) yields
\begin{align}\label{A14}
&\Phi D_1\Phi^{-1}\left(x_ {\beta_1}\otimes\cdots\otimes x_{\beta_k}\right)
=\Phi D_1\left(-x_{\bar{\alpha}_1}\otimes\cdots\otimes x_{\bar{\alpha}_k} + x_ {\beta_1}\otimes\cdots\otimes x_{\beta_k}\right)\notag\\
&=\sum_{j=1}^{k}\sum_{l=1}^{d+1}\left(-(A_1)_{\bar{\alpha}_jl}\Phi(x_{\bar{\alpha}_1}\otimes\cdots\otimes x_l\otimes\cdots\otimes x_{\bar{\alpha}_k}) + (A_1)_{\beta_jl}\Phi(x_{\beta_1}\otimes\cdots\otimes x_l\otimes\cdots\otimes x_{\beta_k})\right)\notag\\
&=\sum_{j=1}^{k}\bigg(-\sum_{l=1}^{d+1}(A_1)_{\bar{\alpha}_jl}x_{\bar{\iota}(\overline{\alpha};j,l)_1}\otimes\cdots\otimes x_{\bar{\iota}(\overline{\alpha};j,l)_k} - \sum_{\substack{l<\bar{\alpha}_{j-1}\\l>\bar{\alpha}_{j+1}}}(A_1)_{\bar{\alpha}_jl} x_{\bar{\alpha}_1}\otimes\cdots\otimes x_l\otimes\cdots\otimes x_{\bar{\alpha}_k}\notag\\
&\quad+\sum_{l=1}^{d+1}(A_1)_{\beta_jl}x_{\bar{\iota}(\beta;j,l)_1}\otimes\cdots\otimes x_{\bar{\iota}(\beta;j,l)_k} + \sum_{\iota(\beta;j,l)\notin\tilde{\mathcal{I}}_{d+1}^{[[k]]}}(A_1)_{\beta_jl}x_{\beta_1}\otimes\cdots\otimes x_l\otimes\cdots\otimes x_{\beta_k} \bigg).
\end{align}
Since $\beta\sim \overline{\alpha}$, we have
\begin{equation*}
  \sum_{j=1}^{k}\sum_{l=1}^{d+1}(A_1)_{\bar{\alpha}_jl}x_{\bar{\iota}(\overline{\alpha};j,l)_1}\otimes\cdots\otimes x_{\bar{\iota}(\overline{\alpha};j,l)_k} =\sum_{j=1}^{k}\sum_{l=1}^{d+1}(A_1)_{\beta_jl}x_{\bar{\iota}(\beta;j,l)_1}\otimes\cdots\otimes x_{\bar{\iota}(\beta;j,l)_k}.
\end{equation*}
The remaining terms in (\ref{A14}) lie in $(\bar{V}^k)^\perp$. Consequently, we derive $\Phi D_1 \Phi^{-1}\Big((\bar{V}^k)^\perp\Big)=(\bar{V}^k)^\perp$, which corresponds to the second block row of (\ref{A9}).

Combining both cases, we conclude that $D_1$ admits a block upper triangular matrix representation of the form (\ref{A9}) with respect to the basis $\{\Phi(x_{\alpha_1}\otimes\cdots\otimes x_{\alpha_k}):\alpha\in\mathcal{I}_{d+1}^{[[k]]}\}$, which completes the proof.
\end{proof}
Since (\ref{A6}) and (\ref{A9}) give matrix representations of the operator $D_1$ on the tensor space $V^k$ under different bases,  the results on diagonalization and spectral properties stated in Lemma \ref{LemA3} and Lemma \ref{LemA4} also apply to the matrix (\ref{A9}). Consequently, from the block upper triangular structure established in Lemma \ref{LemA6}, it follows that $A_1^{[[k]]}$ is diagonalizable and all its eigenvalues have zero real part. Together with the earlier analysis of the structure and similarity properties of $A_1^{[k]}(\hat{\mathbf{x}})$, Proposition~\ref{Lem3-5} is proved.

As a direct corollary of Proposition~\ref{Lem3-5}, we obtain the boundedness of the following matrix exponential, which plays a crucial role in the subsequent convergence analysis.
\begin{Lemma}\label{LemAA2}
Let $\hat{\mathbf{x}}$ be a reference point in a bounded domain independent of $\varepsilon$. The matrix $A_1^{[k]}(\hat{\mathbf{x}})$ is generated in the local linear extension system (\ref{LLES01}). Then
\begin{equation*}
  \left\|\mathrm{e}^{\frac{1}{\varepsilon}A_1^{[k]}(\hat{\mathbf{x}})t}\right\|\leq C
\end{equation*}
holds uniformly for any $\varepsilon,t$ and $\hat{\mathbf{x}}$.
\end{Lemma}

\subsection{Invariant subspaces and block lower triangular diagonalization}\label{SSec3-3}
This subsection studies the algebraic properties of diagonalizing the matrix $A_1^{[k]}(\hat{\mathbf{x}})$ via a block lower triangular transformation (see Proposition~\ref{LemB3} below). These properties underlie the adiabatic transformations used in the proof of Theorem~\ref{Thm3-10}.

For multi-indices $\beta\in\mathcal{I}_{d+1}^{[[j^\prime]]}$ and $\gamma\in\mathcal{I}_{d+1}^{[[j]]}$, we define their concatenation as $$\beta\oplus\gamma:=(\beta_1,\cdots,\beta_{j^\prime}, \gamma_1,\cdots,\gamma_j)\in\mathcal{I}_{d+1}^{[[j^\prime+j]]}.$$
Consider multi-indices of the form $\alpha=\beta\oplus\{d+1\}^{\oplus j}\in\mathcal{I}_{d+1}^{[k]}$ such that $\beta\in\mathcal{I}_d^{[[j^\prime]]}$. Substituting the form of $A_1$ into (\ref{A3}) gives the action of $D_1$
\begin{equation}\label{B2}
	\begin{split}
  D_1((\mathbf{x}-\hat{\mathbf{x}})^\alpha)=&  \sum_{l=1}^{j^\prime}\sum_{m=1}^{d}A_{\beta_lm}(\mathbf{u}-\hat{\mathbf{u}})^{\chi(\beta;l)}(u_m-\hat{u}_m)(t-\hat{t})^j \\
  & +\sum_{l=1}^{j^\prime}\sum_{m=1}^{d}A_{\beta_lm}\hat{u}_m(\mathbf{u}-\hat{\mathbf{u}})^{\chi(\beta;l)}(t-\hat{t})^j.
  \end{split}
\end{equation}
A lemma on $D_1$-invariant subspaces follows directly from (\ref{B2}).
\begin{Lemma}\label{LemB1}
Define the subspace $$V_{j^\prime,j}:=\operatorname{span}\{(\mathbf{x}-\hat{\mathbf{x}})^\alpha:\alpha=\beta\oplus\{d+1\}^{\oplus j},\beta\in\mathcal{I}_d^{[[j^\prime]]}\}.$$

(\romannumeral1) For any integer $m = 0,\dots, k-j$, the direct sum
$\bigoplus\limits_{j^\prime=0}^m V_{j^\prime,j}$
is a $D_1$-invariant subspace.

(\romannumeral2) When $\hat{\mathbf{x}} = \mathbf{0}_{(d+1)\times 1}$, all subspaces $V_{j',j}$ with $0 \leq j' + j \leq k$ are $D_1$-invariant.
\end{Lemma}

Among the invariant subspaces described in part (\romannumeral1) of Lemma \ref{LemB1}, we focus on the subspace $$W_k=\bigoplus\limits_{j^\prime=0}^{1}\bigoplus\limits_ {j=0}^{k-j^\prime}V_{j^\prime,j},$$ which is spanned by the basis $P_{\mathfrak{u}}^{[k,\hat{\mathbf{x}}]}:= \bigcup\limits_{j^\prime=0}^{1}\bigcup\limits_{j=0}^{k-j^\prime}\{(\mathbf{u}-\hat{\mathbf{u}})^{j^\prime}(t-\hat{t})^j\}$. By imposing an appropriate ordering on $P_{\mathfrak{u}}^{[k,\hat{\mathbf{x}}]}$, we obtain an $(k(d+1)+1)$-dimensional vector as
\begin{equation}\label{B3}
\begin{split}
  \mathfrak{u}^{[k,\hat{\mathbf{x}}]}:=&\left[1,\left(\mathbf{u}-\hat{\mathbf{u}}\right)^\top,t-\hat{t},\left(\mathbf{u}-\hat{\mathbf{u}}\right)^\top(t-\hat{t}), \cdots\right. \\
  &\quad\left.\cdots(t-\hat{t})^{k-1},\left(\mathbf{u}-\hat{\mathbf{u}}\right)^\top(t-\hat{t})^{k-1},(t-\hat{t})^k\right]^\top.
\end{split}
\end{equation}

\begin{Lemma}\label{LemB2}
Order the basis elements of $P_{\mathfrak{u}}^{[k,\hat{\mathbf{x}}]}$ as in (\ref{B3}). Then the matrix representation $\bar{\mathcal{A}}_1^{[k]}(\hat{\mathbf{x}})$ of the restricted operator $\left.D_1\right|_{W_k}$ in this basis admits the block-diagonal form
\begin{equation}\label{B4}
  \bar{\mathcal{A}}_1^{[k]}(\hat{\mathbf{x}})=\operatorname{diag}\left\{I_k\otimes\bar{\mathcal{A}}_1(\hat{\mathbf{x}}),0\right\}~\text{with}~\bar{\mathcal{A}}_1(\hat{\mathbf{x}})=
  \left[ \begin{array}{cc} 0 & \mathbf{0}_{1\times d} \\ A\hat{\mathbf{u}} & A \end{array} \right].
\end{equation}
Moreover, if $Q$ diagonalizes $A$ so that $Q^{-1}AQ=\Lambda$, then
\begin{equation}\label{B5}
  Q_{\mathcal{A}}^{[k]}(\hat{\mathbf{x}})=\operatorname{diag}\left\{I_k\otimes Q_{\mathcal{A}}(\hat{\mathbf{x}}),1\right\}~\text{with}~Q_{\mathcal{A}}(\hat{\mathbf{x}})=
  \left[ \begin{array}{cc} 1 &  \\ -\hat{\mathbf{u}} & Q \end{array} \right]
\end{equation}
diagonalizes $\bar{\mathcal{A}}_1^{[k]}(\hat{\mathbf{x}})$ into the canonical form $\operatorname{diag}\{I_k\otimes\operatorname{diag}\{0,\Lambda\},0\}$.
\end{Lemma}

\begin{proof}
Consider multi-indices of the form $\alpha=l\oplus\{d+1\}^{\oplus j}(l\in\mathcal{I}_d^{[[1]]},j=0,\cdots,k-1)$, corresponding to basis elements $(\mathbf{x}-\hat{\mathbf{x}})^\alpha=(u_l-\hat{u}_l)(t-\hat{t})^j\in V_{1,j}$. Then, (\ref{B2}) simplifies to
\begin{equation*}
  D_1((\mathbf{x}-\hat{\mathbf{x}})^\alpha)= \sum_{m=1}^{d}A_{lm}(u_m-\hat{u}_m)(t-\hat{t})^j + \sum_{m=1}^{d}A_{lm}\hat{u}_m(t-\hat{t})^j.
\end{equation*}
When $\alpha=\{d+1\}^{\oplus j}$, we obtain the null action $D_1((\mathbf{x}-\hat{\mathbf{x}})^\alpha)=0$. These computations collectively yield the compact operator representation
\begin{equation*}
\begin{aligned}
  D_1\left(\left[ \begin{array}{c} (t-\hat{t})^j \\ (\mathbf{u}-\hat{\mathbf{u}})(t-\hat{t})^j \end{array} \right]\right)
  &= \bar{\mathcal{A}}_1(\hat{\mathbf{x}})
    \left[ \begin{array}{c} (t-\hat{t})^j \\ (\mathbf{u}-\hat{\mathbf{u}})(t-\hat{t})^j \end{array} \right],\quad j=0,\cdots,k-1,\\
  D_1((t-\hat{t})^k) &=0,
\end{aligned}
\end{equation*}
where $\bar{\mathcal{A}}_1(\hat{\mathbf{x}})$ is defined in (\ref{B4}). Consequently, the matrix representation decomposes as~(\ref{B4}) under the ordered basis $P_{\mathfrak{u}}^{[k,\hat{\mathbf{x}}]}$ specified in (\ref{B3}). The diagonalization property follows from a direct computation and the tensor product formulation.
\end{proof}

We define the set of $j$-th degree polynomials $P_{\mathfrak{v}}^{[[j,\hat{\mathbf{x}}]]}=P_{\mathbf{x}}^{[[j,\hat{\mathbf{x}}]]}\setminus P_{\mathfrak{u}}^{[k,\hat{\mathbf{x}}]}(j=0,\cdots,k)$ via the relative difference of the polynomial sets. It is readily verified that $P_{\mathfrak{v}}^{[[0,\hat{\mathbf{x}}]]}=P_{\mathfrak{v}}^{[[1,\hat{\mathbf{x}}]]}=\emptyset$. For $ j=2,\cdots,k$, the polynomials in $P_{\mathfrak{v}}^{[[j,\hat{\mathbf{x}}]]}$ are arranged in any prescribed order to form the vector $\mathfrak{v}^{[[j,\hat{\mathbf{x}}]]}$. Then we define the local linear extension variable under a new ordering as
\begin{equation}\label{B6}
  \mathbf{x}^{[k,\hat{\mathbf{x}}]}=\left[ \begin{array}{cc} \left(\mathfrak{u}^{[k,\hat{\mathbf{x}}]}\right)^\top & \left(\mathfrak{v}^{[k,\hat{\mathbf{x}}]}\right)^\top \end{array} \right]^\top,~\text{with}~
  \mathfrak{v}^{[k,\hat{\mathbf{x}}]}=\left[ \begin{array}{ccc} \left(\mathfrak{v}^{[[2,\hat{\mathbf{x}}]]}\right)^\top & \cdots & \left(\mathfrak{v}^{[[k,\hat{\mathbf{x}}]]}\right)^\top  \end{array} \right]^\top.
\end{equation}
While this ordering differs from (\ref{LLEV03}) used for numerical constructions, their essential equivalence is claimed in Remark \ref{Rmk0}, allowing our employment of the same notation $\mathbf{x}^{[k,\hat{\mathbf{x}}]}$. We continue to denote by $A_1^{[k]}(\hat{\mathbf{x}})$ the matrix representation of $D_1$ under this ordered basis.

\begin{Proposition}\label{LemB3}
Assume the ordering scheme for $\mathbf{x}^{[k,\hat{\mathbf{x}}]}$ is defined by (\ref{B3}) and (\ref{B6}). Then the matrix $A_1^{[k]}(\hat{\mathbf{x}})$ has a block lower triangular form:
\begin{equation*}
  A_1^{[k]}(\hat{\mathbf{x}})=\left[\begin{array}{cc}
    \bar{\mathcal{A}}_1^{[k]}(\hat{\mathbf{x}}) & \\
    * & *
  \end{array}\right].
\end{equation*}
Furthermore, $A_1^{[k]}(\hat{\mathbf{x}})$ can be diagonalized via the similarity transformation
$$\Lambda_1^{[k]}=(Q^{[k]}(\hat{\mathbf{x}}))^{-1}A_1^{[k]}(\hat{\mathbf{x}})Q^{[k]}(\hat{\mathbf{x}}),$$
with the transformation matrix $Q^{[k]}(\hat{\mathbf{x}})$ satisfying the following property: Both $Q^{[k]}(\hat{\mathbf{x}})$ and $(Q^{[k]}( \hat{\mathbf{x}}))^{-1}$ are block lower triangular, with structures given respectively by
\begin{equation}\label{B7}
  Q^{[k]}(\hat{\mathbf{x}})=\left[\begin{array}{cc}
    Q_{\mathcal{A}}^{[k]}(\hat{\mathbf{x}}) &  \\
    Q_{\mathcal{B}}^{[k]}(\hat{\mathbf{x}}) & Q_{\mathcal{C}}^{[k]}(\hat{\mathbf{x}})
  \end{array}\right], \, (Q^{[k]}(\hat{\mathbf{x}}))^{-1}=\left[\begin{array}{cc}
    \left(Q_{\mathcal{A}}^{[k]}(\hat{\mathbf{x}})\right)^{-1} &  \\
    \bar{Q}_{\mathcal{B}}^{[k]}(\hat{{\mathbf{x}}}) & \left(Q_{\mathcal{C}}^{[k]}(\hat{\mathbf{x}})\right)^{-1}
  \end{array}\right],
\end{equation}
where $Q_ {\mathcal{B}}^{[k]}(\hat{\mathbf{x}}),\bar{Q}_{\mathcal{B}}^{[k]}(\hat{\mathbf{x}}) \in(\mathbb{P}^{k}[\hat{\mathbf{x}}]/\mathbb{P}^{0}[\hat{\mathbf{x}}])^{(D^{[k]}-k(d+1)-1)\times(k(d+1)+1)}$, $Q_{\mathcal{A}}^{[k]}(\hat{\mathbf{x}})$ is defined as in (\ref{B5}), and $Q_{\mathcal{C}}^{[k]}(\hat{\mathbf{x}})$ is a matrix of size $(D^{[k]}-k(d+1)-1)\times(D^{[k]}-k(d+1)-1)$.
\end{Proposition}

\begin{proof}
The structure of $A_1^{[k]}(\hat{\mathbf{x}})$ follows from the matrix representation analysis of $D_1$-invariant subspaces in Lemma \ref{LemB2}, together with the variable ordering specified in (\ref{B6}).

We now derive the block structure (\ref{B7}). By conclusion (\romannumeral2) of Lemma \ref{LemB1}, the subspaces $\operatorname{span}\{P_{\mathfrak{v}}^{[[j,\mathbf{0}]]}\}(j=2,\cdots,k)$ are also $D_1$-invariant when $\hat{\mathbf{x}}=\mathbf{0}_{(d+1)\times 1}$. As a result, under the basis $P_{\mathfrak{u}}^{[k,\hat{\mathbf{x}}]}\cup P_{\mathfrak{v}}^{[k,\mathbf{0}]}$, where $P_{\mathfrak{v}}^{[k,\mathbf{0}]}=\bigcup\limits_{j=2}^k P_{\mathfrak{v}}^{[[j,\mathbf{\mathbf{0}}]]}$, the operator $D_1$ admits a block diagonal matrix representation $\operatorname{diag}\{\bar{\mathcal{A}}_1^{[k]}(\hat{\mathbf{x}}),*\}$. This representation is clearly diagonalizable by a block diagonal matrix $Q_1^{[k]}(\hat{\mathbf{x}})=\operatorname{diag}\{Q_{\mathcal{A}}^{[k]}(\hat{\mathbf{x}}),*\}$ according to Lemma \ref{LemB2}.

It remains to examine the transition matrix from the basis $P_{\mathfrak{u}}^{[k,\hat{\mathbf{x}}]}\cup P_{\mathfrak{v}}^{[k,\mathbf{0}]}$ to $P_{\mathbf{x}}^{[k,\hat{\mathbf{x}}]}$. Since the polynomials in $\mathfrak{v}^{[k,\hat{\mathbf{x}}]}$ are arranged in increasing order by degree according to (\ref{B6}), the same analytical approach as in Lemma \ref{LemAA1} shows that {this transition matrix is lower triangular with the form
\begin{equation*}
  Q_2^{[k]}(\hat{\mathbf{x}})=\left[\begin{array}{cc}
   I_{k(d+1)+1} &  \\
   * & *
   \end{array}\right].
\end{equation*}
Each entry of the lower-left subblock is a polynomial in $\hat{\mathbf{x}}$ whose constant term vanishes. Setting $Q^{[k]}(\hat{\mathbf{x}})=Q_2^{[k]}(\hat{\mathbf{x}})Q_1^{[k]}(\hat{\mathbf{x}})$ then yields the similarity transformation matrix that fulfills all the required properties.}
\end{proof}

The following result follows from direct matrix computations.
\begin{Lemma}\label{LemB4}
Under the ordering (\ref{B3}) and (\ref{B6}), the following identities hold:
\begin{equation}\label{B9}
  \Pi_{j(d+1)+1}^{(j+1)(d+1)}\big(Q^{[k]}(\hat{\mathbf{x}})\big)^{-1} =\left(Q_\mathcal{A}(\hat{\mathbf{x}})\right)^{-1}\Pi_{j(d+1)+1}^{(j+1)(d+1)},\quad j=0,\cdots k-1.
\end{equation}
In particular, we have
\begin{align}
  &\Pi_{j(d+1)+2}^{(j+1)(d+1)}\big(Q^{[k]}(\hat{\mathbf{x}})\big)^{-1}\mathbf{x}^{[k,\hat{\mathbf{x}}]}=(t-\hat{t})^jQ^{-1}\mathbf{u},\quad j=0,\cdots k-1, \label{B10} \\
  &{\Pi_{j(d+1)+1}^{j(d+1)+1}\big(Q^{[k]}(\hat{\mathbf{x}})\big)^{-1}\mathbf{R}^{[k]}(\mathbf{x}^{[k,\hat{\mathbf{x}}]};\hat{\mathbf{x}})=0,\quad j=0,\cdots k.} \label{B11}
\end{align}
\end{Lemma}

\begin{proof}
We first prove (\ref{B9}). Utilizing the block structures presented in (\ref{B5}) and (\ref{B7}), we compute
\begin{align*}
  &\Pi_{j(d+1)+1}^{(j+1)(d+1)}\left(Q^{[k]}(\hat{\mathbf{x}})\right)^{-1}\\
  &=\left[\begin{array}{ccc}
     \textbf{0}_{(d+1)\times(d+1)j} & I_{d+1} & \textbf{0}_{(d+1)\times(D^{[k]}-(d+1)(j+1))}
  \end{array}\right]
  \left[\begin{array}{cccc}
     \left(Q_\mathcal{A}(\hat{\mathbf{x}})\right)^{-1} &  &  & \\
      & \ddots  &  & \\
     &  &  \left(Q_\mathcal{A}(\hat{\mathbf{x}})\right)^{-1} &  \\
     * & \cdots & * & *
  \end{array}\right]\\
  &=\left[\begin{array}{ccc}
     \textbf{0}_{(d+1)\times(d+1)j} & \left(Q_\mathcal{A}(\hat{\mathbf{x}})\right)^{-1} & \textbf{0}_{(d+1)\times(D^{[k]}-(d+1)(j+1))}
  \end{array}\right]
  =\left(Q_\mathcal{A}(\hat{\mathbf{x}})\right)^{-1}\Pi_{j(d+1)+1}^{(j+1)(d+1)}.
\end{align*}
Now, from (\ref{B3}), (\ref{B5}) and~(\ref{B6}), we have
\begin{equation}\label{B12}
(Q_\mathcal{A}(\hat{\mathbf{x}}))^{-1}\Pi_{j(d+1)+1}^{(j+1)(d+1)}\textbf{x}^{[k,\hat{\textbf{x}}]}=
\left[\begin{array}{cc}
      1 &  \\
      \hat{\textbf{u}} & Q^{-1}
  \end{array} \right]
  \left[\begin{array}{c}
    (t-\hat{t})^j  \\
    (\textbf{u}-\hat{\textbf{u}})(t-\hat{t})^j
  \end{array}\right]=\left[\begin{array}{c}
    (t-\hat{t})^j  \\
    (t-\hat{t})^jQ^{-1}\textbf{u}
  \end{array}\right].
\end{equation}
Combined with (\ref{B9}), this directly yields (\ref{B10}).

Next, we note that by the ordering defined in (\ref{B3}), the $(j(d+1)+1)$-th row of $\mathbf{x}^{[k,\hat{\mathbf{x}}]}$ corresponds to the purely time-dependent polynomial $(t-\hat{t})^j$ (for $j=0,\cdots,k$), whose associated multi-index is $\{d+1\}^{\oplus j}$. It then follows from (\ref{LLES02}) that $R^{[k]}_{j(d+1)+1}(\textbf{x}^{[k,\hat{\textbf{x}}]})=0$. Finally, employing (\ref{B9}), identity (\ref{B11}) is verified through a direct computation analogous to that in (\ref{B12}).
\end{proof}

\section{Convergence analysis}\label{Sec4}
For notational simplicity, we denote $\mathbf{x}(t_n)$ by $\mathbf{x}_n$ for $n=0,\cdots,N$. Consider the local linear extension variable at the solution points $\mathbf{x}_n$, together with the associated extension system on the interval $[t_n,t_{n+1}]$:
\begin{equation}\label{LLES-analy}
\begin{aligned}
  &\frac{\mathrm{d}\mathbf{x}^{[k,\mathbf{x}_n]}}{\mathrm{d}t} =\frac{1}{\varepsilon} A_1^{[k]}(\mathbf{x}_n)\mathbf{x}^{[k,\mathbf{x}_n]} +A_0^{[k]}(\mathbf{x}_n)\mathbf{x}^{[k,\mathbf{x}_n]} +\mathbf{R}^{[k]}(\mathbf{x}^{[k,\mathbf{x}_n]}), \quad t_n\leq t\leq t_{n+1},\\
  &\mathbf{x}^{[k,\mathbf{x}_n]}(t_n)=e_1,
\end{aligned}
\end{equation}
where we write $\mathbf{R}^{[k]}(\mathbf{x}^{[k,\mathbf{x}_n]}):=\mathbf{R}^{[k]}(\mathbf{x}^{[k,\mathbf{x}_n]};\mathbf{x}_n)$ for brevity. Taking $\hat{\mathbf{x}}=\mathbf{x}_n$ in Definition \ref{LLEV}, we have $\Pi_2^{d+2}\mathbf{x}^{[k,\mathbf{x}_n]}(t)=\mathbf{x}(t)-\mathbf{x}_n$. {In particular, it holds that $$\mathbf{x}_{n+1}=\Pi_2^{d+2}\mathbf{x}^{[k,\mathbf{x}_n]}(t_{n+1})+\mathbf{x}_n, \quad n=0,\cdots,N-1,$$
which recovers the exact solution of system (\ref{P}) from the high-dimensional variable.}

The truncated version of (\ref{LLES-analy}) is a linear system given by
\begin{equation}\label{LLES-analy-truncation}
\begin{aligned}
  &\frac{\mathrm{d}\tilde{\mathbf{x}}^{[k,\mathbf{x}_n]}}{\mathrm{d}t} =\frac{1}{\varepsilon}A_1^{[k]}(\mathbf{x}_n)\tilde{\mathbf{x}}^{[k,\mathbf{x}_n]} +A_0^{[k]}(\mathbf{x}_n)\tilde{\mathbf{x}}^{[k,\mathbf{x}_n]}, \quad t_n\leq t\leq t_{n+1},\\
  &\tilde{\mathbf{x}}^{[k,\mathbf{x}_n]}(t_n)=e_1.
\end{aligned}
\end{equation}

We similarly denote the projected solution as $\tilde{\mathbf{x}}_{n+1}:=\Pi_2^{d+2}\tilde{\mathbf{x}}^{[k,\mathbf{x}_n]}(t_{n+1})+\mathbf{x}_n$. With the help of the intermediate system (\ref{LLES-analy-truncation}), the error analysis focuses on the difference between the exact solution $\mathbf{x}_{n+1}$ -- obtained from the extension system (\ref{LLES-analy}) -- and the approximation solution $\mathbf{X}_{n+1}$ -- projected from the truncated system (\ref{LLES-num-truncation}).

\subsection{{General results}}\label{SSec4-1}
This subsection establishes the high-order uniform accuracy of the methods {in the general case without assumptions on the scale of $h$ compared to $\varepsilon$.}

\begin{Theorem}\label{Thm3-7}
{Assume that the solution $\mathbf{u}(t)$ of (\ref{P1}) satisfies the bounded oscillatory energy condition (\ref{bdener}), and let numerical solution $\mathbf{U}_n$ $(n=0,\dots,T/h)$ be computed by the scheme (\ref{LLEEI01})-(\ref{LLEEI03}). Then the following error estimate holds:}
\begin{equation*}
\|\mathbf{U}_n-\mathbf{u}(t_n)\| \leq Ch^{k+1},\quad n=0,\dots,T/h.
\end{equation*}
\end{Theorem}

\begin{proof}
In this proof, we assume that the ordering scheme for the extension variables is given by~(\ref{LLEV03}) and divide the procedure into two parts.
	
\textbf{Local estimate.} Applying the variation-of-constants formula to (\ref{LLES-analy}) and (\ref{LLES-analy-truncation}), we have
\begin{equation}\label{Thm3-3-01}
	\begin{split}
  \mathbf{x}^{[k,\mathbf{x}_n]}(t)=&\mathrm{e}^{\frac{1}{\varepsilon}A_1^{[k]}(\mathbf{x}_n)(t-t_n)}e_1 +\int_{t_n}^t\mathrm{e}^{\frac{1}{\varepsilon}A_1^{[k]}(\mathbf{x}_n)(t-s)}A_0^{[k]}(\mathbf{x}_n)\mathbf{x}^{[k,\mathbf{x}_n]}(s)\mathrm{d}s \\ +&\int_{t_n}^t\mathrm{e}^{\frac{1}{\varepsilon}A_1^{[k]}(\mathbf{x}_n)(t-s)}\mathbf{R}^{[k]}(\mathbf{x}^{[k,\mathbf{x}_n]}(s))\mathrm{d}s,
  \end{split}
\end{equation}
and
\begin{equation}\label{Lem3-6-01}
  \tilde{\mathbf{x}}^{[k,\mathbf{x}_n]}(t)=\mathrm{e}^{\frac{1}{\varepsilon}A_1^{[k]}(\mathbf{x}_n)(t-t_n)}e_1 + \int_{t_n}^{t}\mathrm{e}^{\frac{1}{\varepsilon}A_1^{[k]}(\mathbf{x}_n)(t-s)}A_0^{[k]}(\mathbf{x}_n)\tilde{\mathbf{x}}^{[k,\mathbf{x}_n]}(s)\mathrm{d}s,
\end{equation}
respectively. Subtracting (\ref{Lem3-6-01}) from (\ref{Thm3-3-01}) yields
\begin{align}\label{Thm3-3-03}
  \mathcal{J}_1^{[k]}(t):=& \mathbf{x}^{[k,\mathbf{x}_n]}(t)-\tilde{\mathbf{x}}^{[k,\mathbf{x}_n]}(t)\notag\\ =&\int_{t_n}^{t}\mathrm{e}^{\frac{1}{\varepsilon}A_1^{[k]}(\mathbf{x}_n)(t-s)}A_0^{[k]}(\mathbf{x}_n) \left(\mathbf{x}^{[k,\mathbf{x}_n]}(s)-\tilde{\mathbf{x}}^{[k,\mathbf{x}_n]}(s)\right)\mathrm{d}s \\
  &+\int_{t_n}^{t}\mathrm{e}^{\frac{1}{\varepsilon}A_1^{[k]}(\mathbf{x}_n)(t-s)} \mathbf{R}^{[k]}(\mathbf{x}^{[k,\mathbf{x}_n]}(s))\mathrm{d}s.\notag
\end{align}
The bounded oscillatory energy condition (\ref{bdener}) along with (\ref{P1}) gives $\|\dot{\mathbf{u}}(t)\|\leq C$ on $[0,T]$. Then it follows that
\begin{equation*}
  \|\mathbf{x}(s)-\mathbf{x}(t_n)\|\leq \|\mathbf{u}(s)-\mathbf{u}(t_n)\| + (s-t_n) \leq C(s-t_n) \leq Ch,\quad t_n\leq s\leq t_{n+1}.
\end{equation*}
Consequently, the $j$-th order Taylor remainder terms $\mathbf{r}^j$ (see (\ref{ParDiff04})) and their derivatives satisfy
\begin{equation}\label{Thm3-3-04}
  \left\|\mathbf{r}^j\left(\mathbf{x}(s); \mathbf{x}_n\right)\right\|\leq Ch^j, \quad  \Big\|\frac{\partial \mathbf{r}^j\left(\mathbf{x}(s); \mathbf{x}_n\right)}{\partial \mathbf{x}}\Big\|\leq Ch^{j-1},\quad j = 2,\cdots,k+1.
\end{equation}
{Then we derive
\begin{equation}\label{Thm3-3-02}
  \|\mathbf{R}^{[k]}({\mathbf{x}^{[k,\mathbf{x}_n]}}(s))\|\leq Ch^{k+1}
\end{equation}
through (\ref{LLES02}). Substituting (\ref{Thm3-3-02}) into (\ref{Thm3-3-03}) yields}
\begin{equation*}
  \|\mathcal{J}_1^{[k]}(t_{n+1})\| \leq C \int_{t_n}^{t_{n+1}} \|\mathbf{x}^{[k,\mathbf{x}_n]}(s)-\tilde{\mathbf{x}}^{[k,\mathbf{x}_n]}(s)\| \mathrm{d}s + Ch^{k+2}.
\end{equation*}
The local error estimate follows through the integral form of Gronwall inequality
\begin{equation}\label{Thm3-3-05}
  \|\mathcal{J}_1^{[k]}(t_{n+1})\|=\|\mathbf{x}^{[k,\mathbf{x}_n]}(t_{n+1})-\tilde{\mathbf{x}}^{[k,\mathbf{x}_n]}(t_{n+1})\| \leq Ch^{k+2}.
\end{equation}

\textbf{Global estimate.} Applying the nonsingular transformations in Lemma \ref{LemAA1}, we introduce the variable substitutions $\mathbf{X}^{[k,\mathbf{0}]}=\big(S^{[k]}(\mathbf{X}_n)\big)^{-1}\mathbf{X}^{[k,\mathbf{X}_n]}$ and $\tilde{\mathbf{x}} ^{[k,\mathbf{0}]}=\big(S^{[k]}(\mathbf{x}_n)\big)^{-1}\tilde{\mathbf{x}}^{[k,\mathbf{x}_n]}$ into systems (\ref{LLES-num-truncation}) and (\ref{LLES-analy-truncation}), respectively. This yields the transformed systems
\begin{equation*}
	\begin{split}
  &\frac{\mathrm{d}\mathbf{X}^{[k,\mathbf{0}]}}{\mathrm{d}t} =\frac{1}{\varepsilon} A_1^{[k]}(\mathbf{0}) \mathbf{X}^{[k,\mathbf{0}]} +\bar{A}_0^{[k]}(\mathbf{X}_n)\mathbf{X}^{[k,\mathbf{0}]}, \quad
\mathbf{X}^{[k,\mathbf{0}]}(t_n)=\big(S^{[k]}(\mathbf{X}_n)\big)^{-1}e_1,\\
  &\frac{\mathrm{d}\tilde{\mathbf{x}}^{[k,\mathbf{0}]}}{\mathrm{d}t} =\frac{1}{\varepsilon}A_1^{[k]}(\mathbf{0})\tilde{\mathbf{x}}^{[k,\mathbf{0}]} +\bar{A}_0^{[k]}(\mathbf{x}_n)\tilde{\mathbf{x}}^{[k,\mathbf{0}]},\quad
  \tilde{\mathbf{x}}^{[k,\mathbf{0}]}(t_n)=\big(S^{[k]}(\mathbf{x}_n)\big)^{-1}e_1,
  \end{split}
\end{equation*}
where $\bar{A}_0^{[k]}(\cdot)=\big(S^{[k]}(\cdot)\big)^{-1}A_0^{[k]}(\cdot)S^{[k]}(\cdot)$. Using the variation-of-constants formula and subtracting the two expressions gives
\begin{equation}\label{Thm3-3-06}
  \tilde{\mathbf{x}}^{[k,\mathbf{0}]}(t)-\mathbf{X}^{[k,\mathbf{0}]}(t) =\mathrm{e}^{\frac{1}{\varepsilon}A_1^{[k]}(\mathbf{0})(t-t_n)}\Big(\big(S^{[k]}(\mathbf{x}_n)\big)^{-1}-\big(S^{[k]}(\mathbf{X}_n)\big)^{-1}\Big)e_1 +\mathcal{J}_{2}^{[k]}(t),
\end{equation}
with $\mathcal{J}_{2}^{[k]}(t)$ defined as
\begin{align}\label{Thm3-3-07}
  \mathcal{J}_{2}^{[k]}(t):=&\int_{t_n}^{t}\mathrm{e}^{\frac{1}{\varepsilon}A_1^{[k]}(\mathbf{0})(t-s)} \left(\bar{A}_0^{[k]}(\mathbf{x}_n)\tilde{\mathbf{x}}^{[k,\mathbf{0}]}(s) -\bar{A}_0^{[k]}(\mathbf{X}_n)\mathbf{X}^{[k,\mathbf{0}]}(s)\right)\mathrm{d}s \notag\\
  =&\int_{t_n}^t\mathrm{e}^{\frac{1}{\varepsilon}A_1^{[k]}(\mathbf{0})(t-s)} \left(\bar{A}_0^{[k]}(\mathbf{x}_n)-\bar{A}_0^{[k]}(\mathbf{X}_n)\right)\tilde{\mathbf{x}}^{[k,\mathbf{0}]}(s) \mathrm{d}s \\
  &\quad+\int_{t_n}^t\mathrm{e}^{\frac{1}{\varepsilon}A_1^{[k]}(\mathbf{0})(t-s)} \bar{A}_0^{[k]}(\mathbf{X}_n)(\tilde{\mathbf{x}}^{[k,\mathbf{0}]}(s)-\mathbf{X}^{[k,\mathbf{0}]}(s))\mathrm{d}s.\notag
\end{align}
According to Lemma \ref{LemAA1}, we have both $S^{[k]}(\cdot)$ and its inverse are Lipschitz continuous with Lipschitz constants independent of $\varepsilon$. Since $A_0^{[k]}(\cdot)$ is built from the partial derivatives of $\mathbf{F}$, it also enjoys Lipschitz continuity; consequently $\bar{A}_0^{[k]}(\cdot)$ is Lipschitz continuous with a Lipschitz constant independent of $\varepsilon$. Thus, from (\ref{Thm3-3-06}) and (\ref{Thm3-3-07}), it follows that
\begin{equation*}
  \|\tilde{\mathbf{x}}^{[k,\mathbf{0}]}(t)-\mathbf{X}^{[k,\mathbf{0}]}(t)\|\leq C\|\mathbf{x}_n-\mathbf{X}_n\| + C\int_{t_n}^t \|\tilde{\mathbf{x}}^{[k,\mathbf{0}]}(s)-\mathbf{X}^{[k,\mathbf{0}]}(s)\|\mathrm{d}s.
\end{equation*}
Applying the integral form of Gronwall’s inequality yields the estimate
\begin{equation}\label{Thm3-3-08}
  \|\tilde{\mathbf{x}}^{[k,\mathbf{0}]}(t)-\mathbf{X}^{[k,\mathbf{0}]}(t)\|\leq C\mathrm{e}^{Ct}\|\mathbf{x}_n-\mathbf{X}_n\|,\quad t_n\leq t\leq t_{n+1}.
\end{equation}
Exploiting (\ref{Thm3-3-08}) to bound the second term in (\ref{Thm3-3-07}), we obtain the estimate:
\begin{equation}\label{Thm3-3-09}
  \|\mathcal{J}_{2}^{[k]}(t_{n+1})\|\leq Ch\|\mathbf{x}_n-\mathbf{X}_n\|.
\end{equation}
Note that when $\hat{\mathbf{x}}=\mathbf{0}$, Definition \ref{LLEV} implies that the projection $\Pi_2^{d+2}$ applied to the extension variables recovers the original variables. {As a result, we have $\tilde{\mathbf{x}}_{n+1}=\Pi_2^{d+2}\tilde{\mathbf{x}}^{[k,\mathbf{0}]}(t_{n+1})$ and $\mathbf{X}_{n+1}=\Pi_2^{d+2}\mathbf{X}^{[k,\mathbf{0}]}(t_{n+1})$. Projecting equations (\ref{Thm3-3-06}) on both sides and using the matrix structures (\ref{LemAA1-04}), (\ref{LemAA1-01}) and (\ref{LemAA1-02}) gives
\begin{equation*}
  \tilde{\mathbf{x}}_{n+1}-\mathbf{X}_{n+1}=\mathrm{e}^{\frac{1}{\varepsilon}A_1h}(\mathbf{x}_n-\mathbf{X}_n)+\Pi_2^{d+2}\mathcal{J}_{2}^{[k]}(t_{n+1}).
\end{equation*}
We also project (\ref{Thm3-3-03}) by $\Pi_2^{d+2}$ and combine these results, thereby deriving the recurrence}
\begin{equation*}
  \mathbf{x}_{n+1}-\mathbf{X}_{n+1}=\mathrm{e}^{\frac{1}{\varepsilon}A_1h}(\mathbf{x}_n-\mathbf{X}_n) +\Pi_2^{d+2}\big(\mathcal{J}_{1}^{[k]}(t_{n+1})+\mathcal{J}_{2}^{[k]}(t_{n+1})\big).
\end{equation*}
Solving this recurrence with the initial condition $\mathbf{x}_0=\mathbf{X}_0$ yields
\begin{equation}\label{Thm3-3-11}
  \mathbf{x}_n-\mathbf{X}_n=\sum_{j=1}^{n}\mathrm{e}^{\frac{n-j}{\varepsilon}A_1h}\Pi_2^{d+2}\big(\mathcal{J}_{1}^{[k]}(t_j)+\mathcal{J}_{2}^{[k]}(t_j)\big).
\end{equation}
Substituting the bounds (\ref{Thm3-3-05}) and (\ref{Thm3-3-09}) for $\mathcal{J}_{i}^{[k]}(i=1,2)$ into (\ref{Thm3-3-11}) and invoking Lemma~\ref{LemAA2}, we obtain
\begin{equation*}
  \|\mathbf{x}_n-\mathbf{X}_n\|\leq \sum_{j=1}^{n}\big(Ch\|\mathbf{x}_{j-1}-\mathbf{X}_{j-1}\|+Ch^{k+2}\big).
\end{equation*}
Finally, applying the discrete Gronwall inequality gives the desired estimate.
\end{proof}

\subsection{{Improved results for $h>c_0\varepsilon$}}\label{SSec4-2}
We now establish an improved uniform convergence result for the local linear exponential integrators using time steps larger than the scale of $\varepsilon$. In this subsection, we assume that the ordering scheme for extension variables is given by (\ref{B3}) and~(\ref{B6}).

\begin{Theorem}\label{Thm3-10}
Assume the solution $\mathbf{u}(t)$ of (\ref{P1}) satisfies the bounded oscillatory energy condition (\ref{bdener}). Let $\mathbf{U}_n$ be the numerical solution computed by the scheme~ (\ref{LLEEI01})-(\ref{LLEEI03}). {Then for the large time step satisfying $h>c_0\varepsilon$, where $c_0$ is a constant independent of $\varepsilon$,} we have
\begin{equation*}
\|\mathbf{U}_n-\mathbf{u}(t_n)\| \leq C\varepsilon h^k,\quad n=0,\cdots,T/h.
\end{equation*}
\end{Theorem}

We remark that Proposition \ref{Lem3-5} only guarantees zero real parts for eigenvalues of $A_1^{[k]}$, without excluding additional zero eigenvalues beyond those introduced by $\mathbf{x}^{[[0,\hat{\mathbf{x}}]]}$ (see Lemma \ref{LemA4}). This situation arises particularly when system (\ref{P1}) represents the first-order reformulation of the second-order equation (\ref{OP}), where the conjugate eigenvalue pairs of matrix $A$ inherently result in the existence of additional zero eigenvalues. The non-oscillatory modes associated with these zero eigenvalues present additional analytical challenges. To overcome this difficulty, we separate the oscillatory and non-oscillatory components as the primary methodology in our proof.

\begin{proof}
We first apply the adiabatic transformation to the extension variable as follows:
\begin{equation*}
  \mathbf{w}^{[k,\mathbf{x}_n]}(t)=\mathrm{e} ^{-\frac{1}{\varepsilon}\Lambda_1^{[k]}t}\big(Q^{[k]}(\mathbf{x}_n)\big)^{-1}\mathbf{x}^{[k,\mathbf{x}_n]}(t),
\end{equation*}
{where the diagonal matrix $\Lambda_1^{[k]}$ and block-lower-triangular matrix $Q^{[k]}(\mathbf{x}_n)$ is defined in Proposition \ref{LemB3}.} An analogous transformation is applied to the exact solution $\mathbf{u}(t)$ via $\mathbf{w}(t)=\mathrm{e} ^{-\frac{1}{\varepsilon}\Lambda t} Q^{-1}\mathbf{u}(t)$ with $Q$ and $\Lambda$ given in Lemma \ref{LemB2}. Further, using (\ref{B10}), we obtain
\begin{equation}\label{Thm3-10-03}
  \Pi_2^{d+1}\mathbf{w}^{[k,\mathbf{x}_n]}(t)=\mathrm{e}^{-\frac{1}{\varepsilon}\Lambda t}\Pi_2^{d+1}\big(Q^{[k]}(\mathbf{x}_n)\big)^{-1}\mathbf{x}^{[k,\mathbf{x}_n]}(t) =\mathbf{w}(t).
\end{equation}
Thus the relation between the adiabatic variables $\mathbf{w}^{[k,\mathbf{x}_n]}(t)$ and $\mathbf{w}(t)$ preserves the original correspondence between $\mathbf{x}^{[k,\mathbf{x}_n]}(t)$ and $\mathbf{x}(t)$. Similarly, a high-dimensional adiabatic variable is introduced for the truncated system (\ref{LLES-analy-truncation}) by setting
$$\tilde{\mathbf{w}}^{[k,\mathbf{x}_n]}=\mathrm{e} ^{-\frac{1}{\varepsilon}\Lambda_1^{[k]}t}\big(Q^{[k]}(\mathbf{x}_n)\big)^{-1}\tilde{\mathbf{x}}^{[k,\mathbf{x}_n]}.$$
In these adiabatic variables, equation (\ref{Thm3-3-03}) becomes
\begin{align}\label{Thm3-10-04}
&\mathbf{w}^{[k,\mathbf{x}_n]}(t_{n+1})-\tilde{\mathbf{w}}^{[k,\mathbf{x}_n]}(t_{n+1})\notag\\
&= \int_{t_n}^{t_{n+1}} \mathrm{e}^{-\frac{1}{\varepsilon}\Lambda_1^{[k]}s} \big(Q^{[k]}(\mathbf{x}_n)\big)^{-1}A_0^{[k]}(\mathbf{x}_n)Q^{[k]}(\mathbf{x}_n)
\mathrm{e}^{\frac{1}{\varepsilon}\Lambda_1^{[k]}s} \big(\mathbf{w}^{[k,\mathbf{x}_n]}(s)-\tilde{\mathbf{w}}^{[k,\mathbf{x}_n]}(s)\big) \mathrm{d}s\notag\\
&\quad+\int_{t_n}^{t_{n+1}}\mathrm{e}^{-\frac{1}{\varepsilon}\Lambda_1^{[k]}s} \big(Q^{[k]}(\mathbf{x}_n)\big)^{-1}\mathbf{R}^{[k]}\big(Q^{[k]}(\mathbf{x}_n) \mathrm{e}^{\frac{1}{\varepsilon}\Lambda_1^{[k]}s} \mathbf{w}^{[k,\mathbf{x}_n]}(s)\big)\mathrm{d}s=:\mathrm{H}_1+\mathrm{H}_2.
\end{align}

{In what follows, we bound the second integral $\mathrm{H}_2$ in (\ref{Thm3-10-04}) by dividing it into two cases according to the row index.

\textbf{Case 1: Estimate for the first $k(d+1)+1$ rows in $\mathrm{H}_2$}. We perform a row-wise analysis on these rows corresponding to the variables in $P_{\mathfrak{u}}^{[k,\mathbf{x}_n]}$.} We first consider the $(j(d+1)+1)$-th rows for $j=0,\cdots,k$. Lemma \ref{LemB2} implies that the diagonal entries of $\Lambda$ on these rows are all zero. Using (\ref{B11}), we obtain
\begin{equation}\label{Thm3-10-05}
  \Pi_{j(d+1)+1}^{j(d+1)+1}\mathrm{H}_2=0, \, j=0,\cdots,k.
\end{equation}

For the remaining rows, the analysis is then conducted for rows $2$ through $d+1$, corresponding to the variable $\mathbf{u}$. By (\ref{B9}) and (\ref{LLES02}), the restriction of $\mathrm{H}_2$ to these rows takes the form
\begin{equation}\label{Thm3-10-02}
  \Pi_{2}^{d+1}\mathrm{H}_2=\int_{t_n}^{t_{n+1}}\mathrm{e}^{-\frac{1}{\varepsilon}\Lambda s} Q^{-1}\mathbf{r}_n^{k+1}\big(Q\mathrm{e}^{\frac{1}{\varepsilon}\Lambda s}\mathbf{w}(s),s\big) \mathrm{d}s,
\end{equation}
where $\mathbf{r}_n^{k+1}$ denotes the Taylor remainder (\ref{ParDiff04}) evaluated at $(\hat{\mathbf{u}},\hat{t})=(\mathbf{u}(t_n),t_n)$. Consider the $q$-th row $(1\leq q\leq d)$ of (\ref{Thm3-10-02}), written as
\begin{equation}\label{Thm3-10-06}
  \int_{t_n}^{t_{n+1}}\mathrm{e}^{-\frac{\mathrm{i}}{\varepsilon}\lambda_qs}\beta_q(s)\mathrm{d}s \quad \text{with}\quad
  \beta_q(s):=\sum_{p=1}^{d}\big(Q^{-1}\big)_{qp} r_{n,p}^{k+1}\big(Q\mathrm{e}^{\frac{1}{\varepsilon}\Lambda s}\mathbf{w}(s),s\big).
\end{equation}
Since the oscillator $\mathrm{e}^{-\frac{\mathrm{i}}{\varepsilon}\lambda_qs}$ has a period of $\frac{2\pi\varepsilon}{|\lambda_q|}$, we define $N_q=\left\lfloor\frac{h|\lambda_q|}{2\pi\varepsilon}\right\rfloor$ as the number of complete oscillation cycles within a single time step, which satisfies
\begin{equation}\label{Thm3-10-07}
  \frac{2\pi\varepsilon} {|\lambda_q|}N_q\leq t_{n+1}-t_n \leq \frac{2\pi\varepsilon} {|\lambda_q|}(N_q+1).
\end{equation}
Given the condition $h>c_0\varepsilon$, it is enough to assume $N_q\geq 1$. Then the integral (\ref{Thm3-10-06}) decomposes as
\begin{equation}\label{Thm3-10-08}
	\begin{split}
 \bigg(\sum_{r=1}^{N_q}&\Big( \int_{t_n+\frac{2\pi\varepsilon}{|\lambda_q|}r-\frac{2\pi\varepsilon}{|\lambda_q|}}^{t_n+\frac{2\pi\varepsilon}{|\lambda_q|}r-\frac{\pi\varepsilon}{|\lambda_q|}} +\int_{t_n+\frac{2\pi\varepsilon}{|\lambda_q|}r-\frac{\pi\varepsilon}{|\lambda_q|}}^{t_n+\frac{2\pi\varepsilon}{|\lambda_q|}r} \Big)+\int_{t_n+\frac{2\pi\varepsilon}{|\lambda_q|}N_q}^{t_{n+1}}\bigg) \mathrm{e}^{-\frac{\mathrm{i}}{\varepsilon}\lambda_qs}\beta_q(s)\mathrm{d}s\\
  &=\sum_{r=1}^{N_q} \int_{t_n+\frac{2\pi\varepsilon}{|\lambda_q|}r-\frac{2\pi\varepsilon}{|\lambda_q|}}^{t_n+\frac{2\pi\varepsilon}{|\lambda_q|}r-\frac{\pi\varepsilon}{|\lambda_q|}}
  \mathrm{e}^{-\frac{\mathrm{i}}{\varepsilon}\lambda_qs}\Big(\beta_q(s)-\beta_q\big(s+\frac{\pi\varepsilon}{|\lambda_q|}\big)\Big)\mathrm{d}s\\ &\qquad\qquad\qquad+\int_{t_n+\frac{2\pi\varepsilon}{|\lambda_q|}N_q}^{t_{n+1}} \mathrm{e}^{-\frac{\mathrm{i}}{\varepsilon}\lambda_qs}\beta_q(s)\mathrm{d}s.
  \end{split}
\end{equation}
Applying (\ref{Thm3-3-04}) and the definition of $\beta_q(s)$, it follows from (\ref{Thm3-10-07}) that
\begin{equation}\label{Thm3-10-09}
  \bigg\|\int_{t_n+\frac{2\pi\varepsilon}{|\lambda_q|}N_q}^{t_{n+1}} \mathrm{e}^{-\frac{\mathrm{i}}{\varepsilon}\lambda_qs}\beta_q(s)\mathrm{d}s\bigg\| \leq \Big(t_{n+1}-t_n-\frac{2\pi\varepsilon}{|\lambda_q|}N_q\Big) \|\beta_q(s)\| \leq C\varepsilon h^{k+1}.
\end{equation}

From the original system (\ref{P1}), the adiabatic variable $\mathbf{w}(t)$ satisfies
\begin{equation*}
  \frac{\mathrm{d}\mathbf{w}}{\mathrm{d}t} = \mathrm{e}^{-\frac{1}{\varepsilon}\Lambda t}Q^{-1} \mathbf{F}\big(Q\mathrm{e}^{\frac{1}{\varepsilon}\Lambda t}\mathbf{w}(s),s\big).
\end{equation*}
Then we have
\begin{align*}
&\beta_q(s)-\beta_q\Big(s+\frac{\pi\varepsilon}{|\lambda_q|}\Big) \\
&=-\sum_{p=1}^{d}\big(Q^{-1}\big)_{qp}\int_{0}^{\frac{\pi\varepsilon}{|\lambda_q|}}
\frac{\mathrm{d}}{\mathrm{d}\theta}r_{n,p}^{k+1} \big(Q\mathrm{e}^{\frac{1}{\varepsilon}\Lambda (s+\theta)}\mathbf{w}(s+\theta),s+\theta\big)\mathrm{d}\theta\notag\\
&=-\sum_{p=1}^{d}\big(Q^{-1}\big)_{qp} \int_{0}^{\frac{\pi\varepsilon}{|\lambda_q|}}\Big(\frac{\partial r_{n,p}^{k+1}(\mathbf{x}(s+\theta))}{\partial\mathbf{u}} \big(\frac{1}{\varepsilon}Q\Lambda\mathrm{e}^{\frac{1}{\varepsilon}\Lambda(s+\theta)}\mathbf{w}(s+\theta) +\mathbf{F}(\mathbf{x}(s+\theta)) \big) \\
&\qquad\qquad\qquad\qquad\qquad +\frac{\partial r_{n,p}^{k+1}(\mathbf{x}(s+\theta))}{\partial t} \Big)\mathrm{d}\theta.
\end{align*}
The bounded oscillatory energy condition (\ref{bdener}) guarantees that  $\|\mathbf{w}\|\leq C\varepsilon$. Then the estimate (\ref{Thm3-3-04}) implies{
\begin{equation}\label{Thm3-10-10}
\begin{split}
  \Big\|\beta_q(s)-\beta_q\Big(s+\frac{\pi\varepsilon}{|\lambda_q|}\Big)\Big\|
  \leq & C\int_{0}^{\frac{\pi\varepsilon}{|\lambda_q|}} \Big\|\frac{\partial r_n^{k+1}(\mathbf{x}(s+\theta))}{\partial\mathbf{x}} \Big\| \Big(\frac{1}{\varepsilon}\|\mathbf{w}(s+\theta)\| \\
  & +  \|\mathbf{F}(\mathbf{u}(s+\theta))\|+1\Big)\mathrm{d}\theta \leq C\varepsilon h^k.
  \end{split}
\end{equation}}
From (\ref{Thm3-10-07})--(\ref{Thm3-10-10}), we obtain for all $q=1,\cdots,d$
\begin{align*}
  &\left\|\int_{t_n}^{t_{n+1}}\mathrm{e}^{-\frac{\mathrm{i}}{\varepsilon}\lambda_qs}\beta_q(s)\mathrm{d}s\right\| \leq \sum_{r=1}^{N_q} \int_{t_n+\frac{2\pi\varepsilon}{|\lambda_q|}r-\frac{2\pi\varepsilon}{|\lambda_q|}}^{t_n+\frac{2\pi\varepsilon}{|\lambda_q|}r-\frac{\pi\varepsilon}{|\lambda_q|}}
  C\Big\|\beta_q(s)-\beta_q\big(s+\frac{\pi\varepsilon}{|\lambda_q|}\big)\Big\|\mathrm{d}s \notag\\
 &\qquad +\Big\|\int_{t_n+\frac{2\pi\varepsilon}{|\lambda_q|}N_q}^{t_{n+1}} \mathrm{e}^{-\frac{\mathrm{i}}{\varepsilon}\lambda_qs}\beta_q(s)\mathrm{d}s\Big\|\leq CN_q\frac{\pi\varepsilon^2 h^k}{|\lambda_q|} + C\varepsilon h^{k+1}
  \leq C\varepsilon h^{k+1}.
\end{align*}
Hence, we derive that $\mathrm{H}_2$ in (\ref{Thm3-10-04}), restricted to the rows $1$ through $d+1$, satisfies the bound $O(\varepsilon h^{k+1})$.

The same methodology extends naturally to rows $j(d+1)+1$ through $(j+1)(d+1)$ for $j=1,\cdots,k-1$. In fact, according to (\ref{B10}), we obtain $\Pi_{j(d+1)+2}^{(j+1)(d+1)}\mathbf{w}^{[k,\mathbf{x}_n]}(t)=(t-t_n)^j\mathbf{w}(t)$ analogous to (\ref{Thm3-10-03}). From (\ref{LLES02}), the remainder terms for these rows reduce to $(t-t_n)^j\mathbf{r}_n^{k-j+1}$, which is again bounded by $O(h^{k+1})$. These observations guarantee that this error result holds for all the first $(k(d+1)+1)$ rows:
\begin{equation}\label{Thm3-10-11}
  \big\|\Pi_1^{k(d+1)+1}\mathrm{H}_2\big\|\leq C\varepsilon h^{k+1}.
\end{equation}

{\textbf{Case 2: Estimate for the remaining rows in $\mathrm{H}_2$.} Substituting the block structure of $(Q^{[k]}(\mathbf{x}_n))^{-1}$ from (\ref{B7}) yields the relation
\begin{equation}\label{Thm3-10-12}
\begin{split}
  &\Pi_{k(d+1)+2}^{D^{[k]}}\big(Q^{[k]}(\mathbf{x}_n)\big)^{-1}\mathbf{R}^{[k]}\big(Q^{[k]}(\mathbf{x}_n) \mathrm{e}^{\frac{1}{\varepsilon}\Lambda_1^{[k]}s} \mathbf{w}^{[k,\mathbf{x}_n]}(s)\big) \\
  &\qquad=\bar{Q}_{\mathcal{B}}^{[k]}(\mathbf{x}_n)\Pi_1^{k(d+1)+1}\mathbf{R}^{[k]}+ \big(Q_{\mathcal{C}}^{[k]}(\mathbf{x}_n)\big)^{-1}\Pi_{k(d+1)+2}^{D^{[k]}}\mathbf{R}^{[k]}.
  \end{split}
\end{equation}

Since $\bar{Q}_{\mathcal{B}}^{[k]}$ belongs to the space $(\mathbb{P}^{k}[\mathbf{x}_n]/\mathbb{P}^{0}[\mathbf{x}_n])^{(D^{[k]}-k(d+1)-1)\times(k(d+1)+1)}$ (see Lemma~\ref{LemB3}), it follows from (\ref{bdener}) that $\bar{Q} _{\mathcal{B}}^{[k]}(\mathbf{x}_n)=O(\varepsilon)$.  Using the estimate (\ref{Thm3-3-02}), the first term in (\ref{Thm3-10-12}) is shown to be of order $O(\varepsilon h^{k+1})$. The projection operator in the second term corresponds to variables in $P_{\mathfrak{v}}^{[k,\mathbf{x}_n]}$. According to (\ref{B3}) and (\ref{B6}), the multi-indices associated with these polynomial variables contain at least two components distinct from $d+1$. Hence, it follows from (\ref{LLES02}), (\ref{Thm3-3-04}) and  (\ref{bdener}) that $$\|\Pi_{k(d+1)+2}^{D^{[k]}}\mathbf{R} ^{[k]}\|\leq C\varepsilon h^k.$$
Combining the estimate of the terms in (\ref{Thm3-10-12}) yields
\begin{equation*}
  \big\|\Pi_{k(d+1)+2}^{D^{[k]}}\mathrm{H}_2\big\|\leq C\varepsilon h^{k+1}.
\end{equation*}
Together with (\ref{Thm3-10-11}), this result shows that
\begin{equation}\label{Thm3-111}
  \big\|\mathrm{H}_2\big\|\leq C\varepsilon h^{k+1}.
\end{equation}

Finally, using \eqref{Thm3-111}, we obtain the following estimate for the second integral in (\ref{Thm3-3-03})
\begin{equation*}
  \left\|\int_{t_n}^{t_{n+1}}\mathrm{e}^{\frac{1}{\varepsilon}A_1^{[k]}(\mathbf{x}_n)(t_{n+1}-s)} \mathbf{R}^{[k]}(\mathbf{x}^{[k,\mathbf{x}_n]}(s)) \mathrm{d}s\right\| =\left\|Q^{[k]}(\mathbf{x}_n)\mathrm{e}^{\frac{1}{\varepsilon}\Lambda_1^{[k]} t_{n+1}}\mathrm{H}_2 \right\|  \leq C\varepsilon h^{k+1}.
\end{equation*}
The application of Gronwall's inequality to (\ref{Thm3-3-03}) then gives the improved local error bound $$\|\mathbf{x}^{[k,\mathbf{x}_n]}(t_{n+1})-\tilde{\mathbf{x}}^{[k,\mathbf{x}_n]}(t_{n+1})\| \leq C\varepsilon h^{k+1}.$$}

The argument of global estimate proceeds identically to that in Theorem \ref{Thm3-7} and is therefore omitted.
\end{proof}

\begin{Remark}\label{Rmk6}
When using $k$-th order local linear extension variables, if the sum of any $l$ eigenvalues of $A$ satisfy $\lambda_{j_1}+\cdots+\lambda_{j_l}=O(1)$ for any integer $l$ with \(1 \leq l \leq k\), then Lemma~\ref{LemA4} guarantees that \(A_1^{[k]}\) has no zero eigenvalues. In this non-resonance case, the row-by-row analysis, based on periodic decomposition approach developed for the first \(k(d+1)+1\) rows, can be directly extended to all rows.
\end{Remark}

{We now summarize and interpret the theoretical results of the previous two subsections from the perspective of applying the numerical scheme under different step-size regimes. The convergence order $O(\varepsilon h^k)$ established in Theorem \ref{Thm3-10} for large step sizes improves upon the general order $O(h^{k+1})$ given in Theorem \ref{Thm3-7}. Therefore, the former provides the theoretical foundation for computations with large steps, whereas the latter governs the error in small-step regimes. Consequently, for a fixed $\varepsilon$, the local linear extension exponential integrator exhibits a piecewise convergence behavior in $h$. Conversely, for a fixed step size $h$, the numerical error also depends on $\varepsilon$ in a piecewise manner.

We further examine the location of the transition points between the two regimes. The periodic decomposition technique used in (\ref{Thm3-10-08}) requires the time step to exceed the period of every oscillator $\mathrm{e}^{\frac{\mathrm{i}}{\varepsilon}\lambda_jt}$. This requirement is satisfied when $$h>h_0:=\frac{2\pi\varepsilon}{\mu},$$ where $\mu=\min|\lambda(A)|$ denotes the minimal magnitude of the eigenvalues of $A$. Hence, the value $h=h_0$ generally acts as a practical threshold for the numerical solution $\mathbf{u}_n$ to separate the two convergence orders. These observations will be verified numerically in Section \ref{Sec5}.}

\subsection{Application on second-order ODEs}
Through the variable transformation $\mathbf{p}=\varepsilon \dot{\mathbf{y}}$ and denoting $\mathbf{u}=[\mathbf{y}^\top,\mathbf{p}^\top]^\top$, we reformulate the system (\ref{OP}) as a first-order differential equation in the form (\ref{P1}), where
\begin{equation*}
  A=\left[\begin{array}{cc}
              \mathbf{0}_{d\times d} & I_d\\
              -M & \mathbf{0}_{d\times d}\\
            \end{array}\right],\quad
  \mathbf{F}(\mathbf{u},t)=\left[\begin{array}{c}
                                   \mathbf{0}_{d\times 1} \\
                                   \varepsilon\mathbf{g}(\mathbf{y},t)\\
                                 \end{array}\right],\quad
  \mathbf{u}_{\mathrm{in}}=\left[\begin{array}{c}
                       \mathbf{y}_{\mathrm{in}} \\
                       \varepsilon\dot{\mathbf{y}}_{\mathrm{in}} \\
                     \end{array}\right].
\end{equation*}
When $M$ is a symmetric positive definite matrix, the matrix $A$ is diagonalizable {with its spectrum given by $\lambda(A)=\pm\mathrm{i}\sqrt{\lambda(M)}$}, which comprises $d$ conjugate pairs of purely imaginary eigenvalues. Building upon the convergence results established in Theorems \ref{Thm3-7} and \ref{Thm3-10}, we obtain the following high-order uniform convergence result.
\begin{Theorem}\label{Thm4-7}
Let $\mathbf{y}(t)$ and $\dot{\mathbf{y}}(t)$ be solutions of (\ref{OP}) that satisfy (\ref{boec}). Suppose the derivatives of $\mathbf{g}$ are Lipschitz continuous in both $\mathbf{y}$ and $t$ variables up to order $k$ with Lipschitz constants independent of $\varepsilon$. {Then the uniform error estimate holds:}
\begin{equation*}
  \|\mathbf{y}_n-\mathbf{y}(t_n)\| +\varepsilon\|\dot{\mathbf{y}}_n-\dot{\mathbf{y}}(t_n)\|\leq C\varepsilon h^{k+1}.
\end{equation*}
{In particular, for the large time step satisfying $h > c_0\varepsilon$, the following improved error estimate holds:}
\begin{equation*}
  \|\mathbf{y}_n-\mathbf{y}(t_n)\| +\varepsilon\|\dot{\mathbf{y}}_n-\dot{\mathbf{y}}(t_n)\|\leq C\varepsilon^2 h^{k}.
\end{equation*}
\end{Theorem}

\section{Numerical results}\label{Sec5}
This section presents numerical experiments to verify the sharpness of the error bounds established in Section \ref{Sec4}. The local linear extension exponential integrator employing polynomials of order $k$ will be denoted by LLEEI$(k+1)$. Further numerical tests are provided in \cite{QD2025b}.

\textbf{Example 1.} The linear and nonlinear terms in (\ref{OP}) are taken as $M=1$ and $g(y,t)=-(t+\cos(2\sqrt{6}t))\sin y$, respectively. Simulations run up to the final time $T=6$. Initial conditions are set to $y(0)=\varepsilon$ and $\dot{y}(0)=\sqrt{3}$, which satisfy the requirement of bounded oscillatory energy. Since analytical solutions are unavailable for the numerical example, reference solutions are computed using the classical fourth-order Runge–Kutta method with a fine step size $h=0.1/2^{18}$. Errors are measured via $error(\mathbf{y})$ and $error(\dot{\mathbf{y}})$, where $error(\cdot)$ denotes the global maximum error over the time grid $\{t_n\}_{n=0}^N$.

We first consider the weakly stiff case by fixing $\varepsilon=1/2^2$. The top row of Fig.~\ref{fig1-1} shows how the numerical errors depend on the step size $h$. The results confirm that each LLEEI$(k+1)$ method exhibits $(k+1)$-th order convergence, demonstrating that the order can be increased arbitrarily with $k$. The absence of convergence for LLEEI$5$ and LLEEI$6$ is due to their errors having reached machine precision. To examine the highly oscillatory regime, we choose $\varepsilon=1/2^8$; the corresponding results are shown in the bottom row of Fig.~\ref{fig1-1}. {In each panel, the vertical dotted line marks $h=h_0$, where $h_0=2\pi\varepsilon$ is the threshold for this example.} The plots demonstrate that $k$-th order convergence holds for large step sizes $h > h_0$, while $(k+1)$-th order convergence is achieved when $h < h_0$.

\begin{figure}[htbp]
  \centering
  \includegraphics[width=\textwidth]{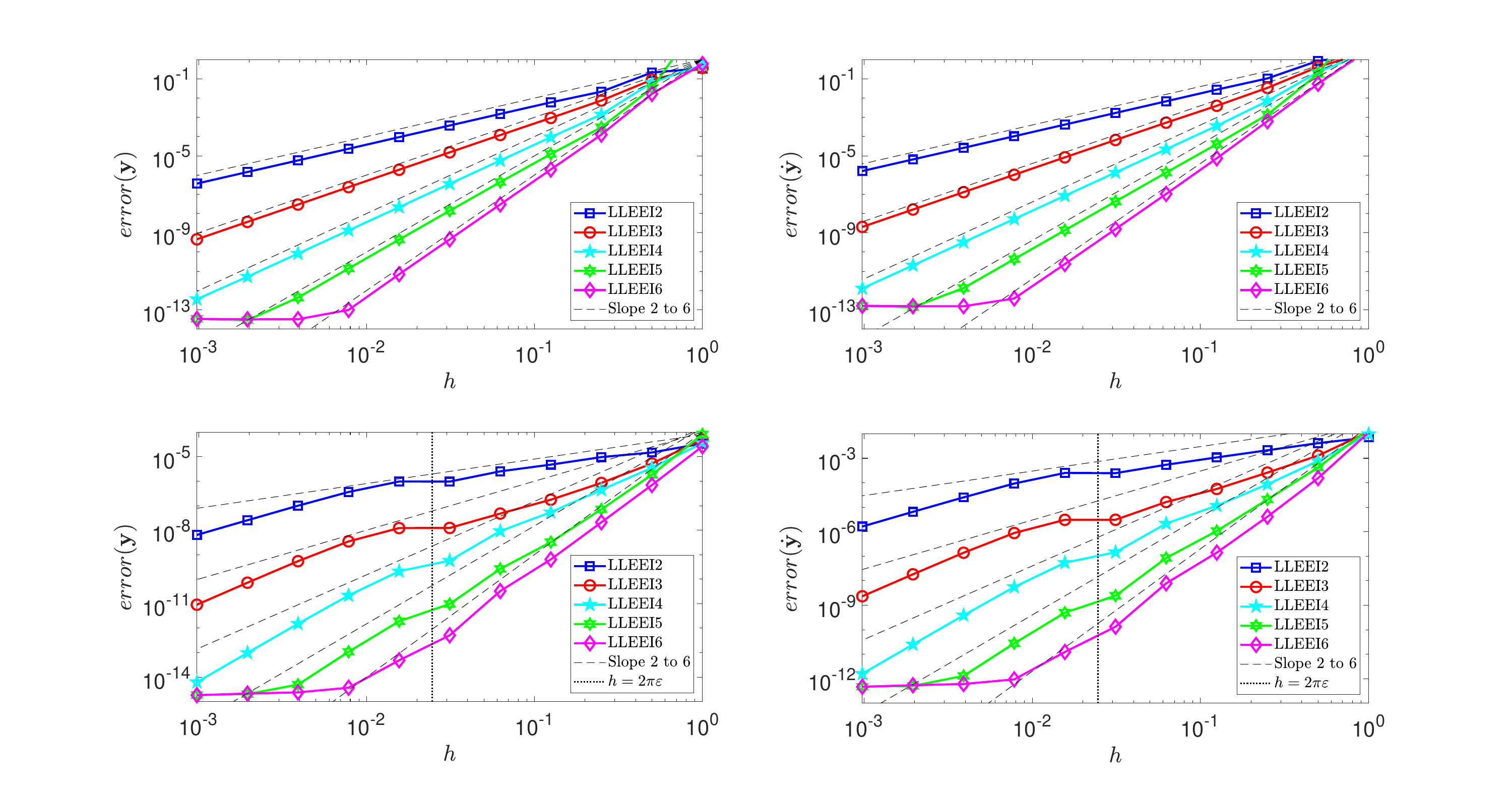}
  \caption{Error versus step size $h$. The top row corresponds to the weakly oscillatory case ($\varepsilon=1/2^2$), while the bottom row represents the highly oscillatory regime ($\varepsilon=1/2^8$). The left and right columns display the numerical errors for $\mathbf{y}$ and $\dot{\mathbf{y}}$, respectively.}\label{fig1-1}
\end{figure}

Next, we examine the uniformity of numerical errors as $\varepsilon$ varies. The top row of Fig.~\ref{fig1-2} presents the errors for different values of $\varepsilon$ with a fixed small step size $h=1/2^6$. {In each panel, the vertical dotted line represents $\varepsilon=\varepsilon_0$. For this specific example, the threshold evaluates to $\varepsilon_0=\frac{h}{2\pi}$, obtained by inverting the critical conditions $h=h_0=\frac{2\pi\varepsilon} {\sqrt{\mu}}$.} When $\varepsilon > \varepsilon_0$ (i.e., $h < h_0$), the LLEEI$(k+1)$ methods yield errors in $\mathbf{y}$ that scale linearly with $\varepsilon$, while the errors in $\dot{\mathbf{y}}$ remain uniform. Finally, we test a large step size $h=1/2$. The corresponding errors for varying $\varepsilon$ are displayed in the bottom row of Fig. \ref{fig1-2}. The results confirm that for $\varepsilon<\varepsilon_0$ (i.e., $h>h_0$), the errors in $\mathbf{y}$ and $\dot{\mathbf{y}}$ exhibit second-order and first-order convergence with respect to $\varepsilon$, respectively, matching the predictions of Theorem \ref{Thm4-7}.

\begin{figure}[htbp]
  \centering
  \includegraphics[width=\textwidth]{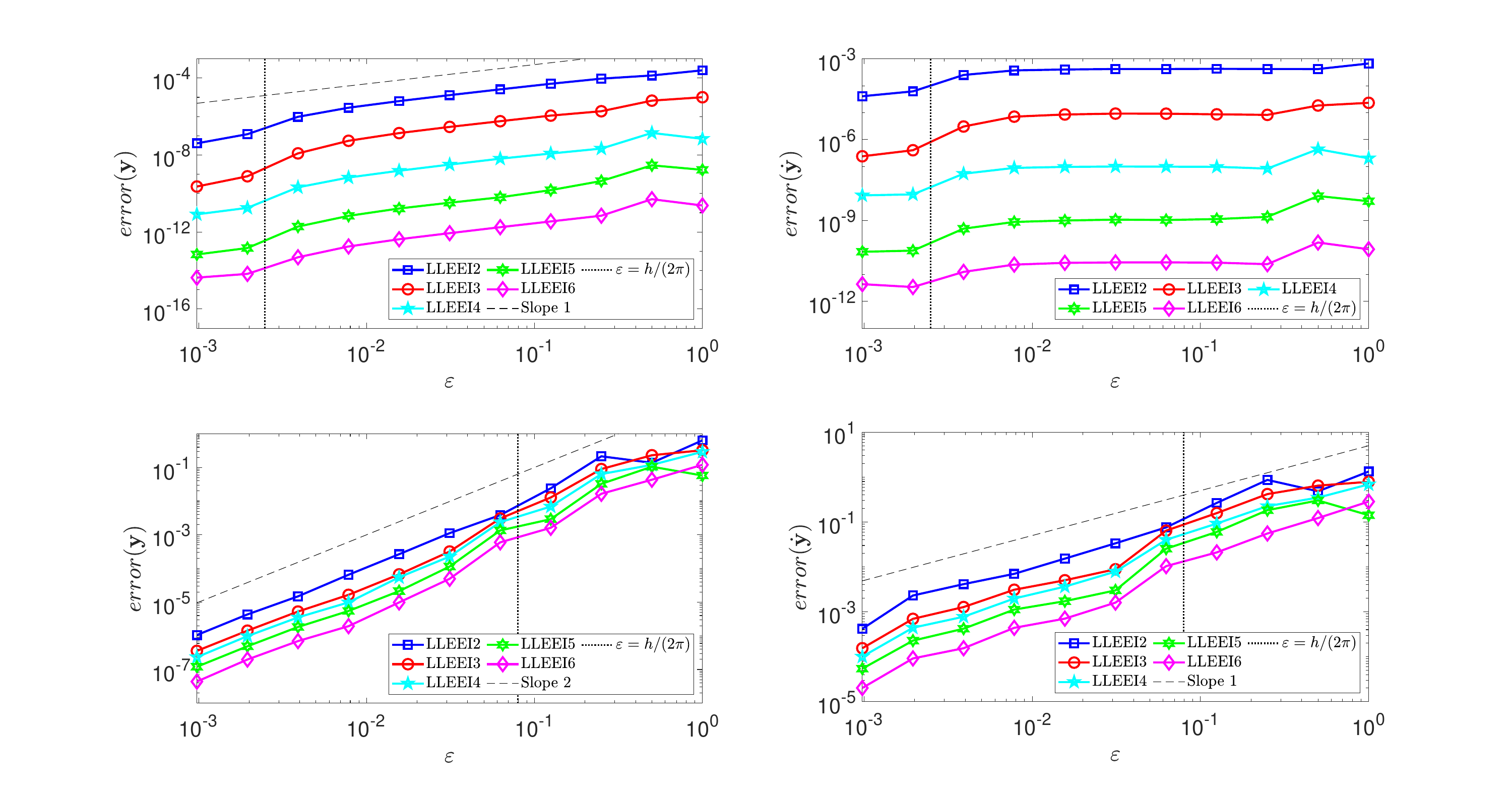}
  \caption{Error versus step size $\varepsilon$. The top row corresponds to the small step size case ($h=1/2^6$), while the bottom row represents the large step size case ($h=1/2$). The left and right columns display the numerical errors for $\mathbf{y}$ and $\dot{\mathbf{y}}$, respectively.}\label{fig1-2}
\end{figure}
\vskip 2mm

\textbf{Example 2.} We consider system (\ref{P1}) with the matrix and nonlinear function given by
\begin{equation}\label{ex2}
  A=\left[\begin{array}{cccc}
      0 & 0 & 1 & 0 \\
      0 & 0 & 0 & 1 \\
      E & 0 & 0 & B \\
      0 & E & -B & 0
    \end{array}\right],\quad
  \mathbf{F}=\left[\begin{array}{c}
               0 \\
               0 \\
               g_1(\mathbf{u})\\
               g_2(\mathbf{u})
             \end{array}\right],
\end{equation}
which models the 2D dynamics of a charged particle in an isotropic strong electric field and a perpendicular uniform strong magnetic field. Applying the scaling $\mathbf{p} = \varepsilon \dot{\mathbf{y}}$, and defining the state variable $\mathbf{u} = [\mathbf{y}^\top, \mathbf{p}^\top]^\top$, the system transforms from the underlying second-order equation
\begin{equation}\label{CPD}
  \ddot{\mathbf{y}}=\frac{E}{\varepsilon^2}\mathbf{y} + \frac{B}{\varepsilon}J\dot{\mathbf{y}} + \frac{1}{\varepsilon}\mathbf{g}(\mathbf{y},t),\quad
  J=\left[\begin{array}{cc}
            0 & 1 \\
            -1 & 0
          \end{array}\right],
\end{equation}
where $E$ and $B$ are the electric and magnetic field intensities, respectively. {We note that despite the presence of first-order derivative terms in system (\ref{CPD}), the oscillatory energy retains the form (\ref{OscEnergy}) since the magnetic force does not alter the particle's mechanical energy.} The particle is initially at the origin, with $\mathbf{y}(0) = [0,0]^\top$ and velocity $\dot{\mathbf{y}}(0) = [3,4]^\top$. The magnetic field is set by $B=1$, and the nonlinear term is given by the electrostatic force
\begin{equation*}
  \mathbf{g}(\mathbf{y},t)=\left[\begin{array}{c}
          \frac{y_1}{(y_1^2+y_2^2+(2-\cos(\pi t))^2)^{\frac{3}{2}}}\\
          \frac{y_2}{(y_1^2+y_2^2+(2-\cos(\pi t))^2)^{\frac{3}{2}}}
        \end{array}\right].
\end{equation*}
Two electric field configurations are examined. For \(E = 6\), the matrix $A$ in (\ref{ex2}) has eigenvalues \(\lambda(A) = \{\pm 2\mathrm{i}, \pm 3\mathrm{i}\}\); the corresponding linear high‑frequency oscillators share a common period. For \(E = 3\), the eigenvalues are \(\lambda(A) = \bigl\{\pm \sqrt{\frac{7\pm\sqrt{13}}{2}}\,\mathrm{i}\bigr\}\), producing non‑resonant oscillators.

For comparison, the following experiments include the second- and third-order exponential time-differencing Runge–Kutta methods from \cite{CM2002} and the fourth-order scheme from \cite{K2005}, labeled ETDRK2, ETDRK3 and ETDRK4, respectively. We also test the $k$-th order exponential Rosenbrock methods from \cite{HOS2009}, denoted EXPRB$k$. Note that EXPRB2 is essentially equivalent to LLEEI2 \cite{QD2025b}. Classical Runge–Kutta methods of order $k$ are referred to as RK$k$. The LLEEI schemes are implemented directly, without requiring advance diagonalization of $A$ or any preprocessing of periodic effects.

Fig. \ref{fig2-1} shows the error dependence of $\mathbf{u}$ on step size $h$ with fixed $\varepsilon=1/2^2$ and $\varepsilon=1/2^8$, respectively. In the weakly oscillatory regime (top row), all methods achieve their theoretical convergence orders, with only minor performance differences among exponential-type integrators. For highly oscillatory cases (bottom row), the LLEEI methods exhibit the same piecewise convergence behavior predicted by Theorems~\ref{Thm3-7} and \ref{Thm3-10}. The ETDRK and EXPRB methods remain stable under stiffness, but their convergence is limited by the oscillatory nature of the solution. Runge-Kutta methods become unstable in this regime and are therefore omitted. We also observe that whether the eigenvalues of matrix $A$ are integer multiples of a common frequency has negligible impact on the performance of exponential-type methods.

\begin{figure}[htbp]
  \centering
  \includegraphics[width=\textwidth]{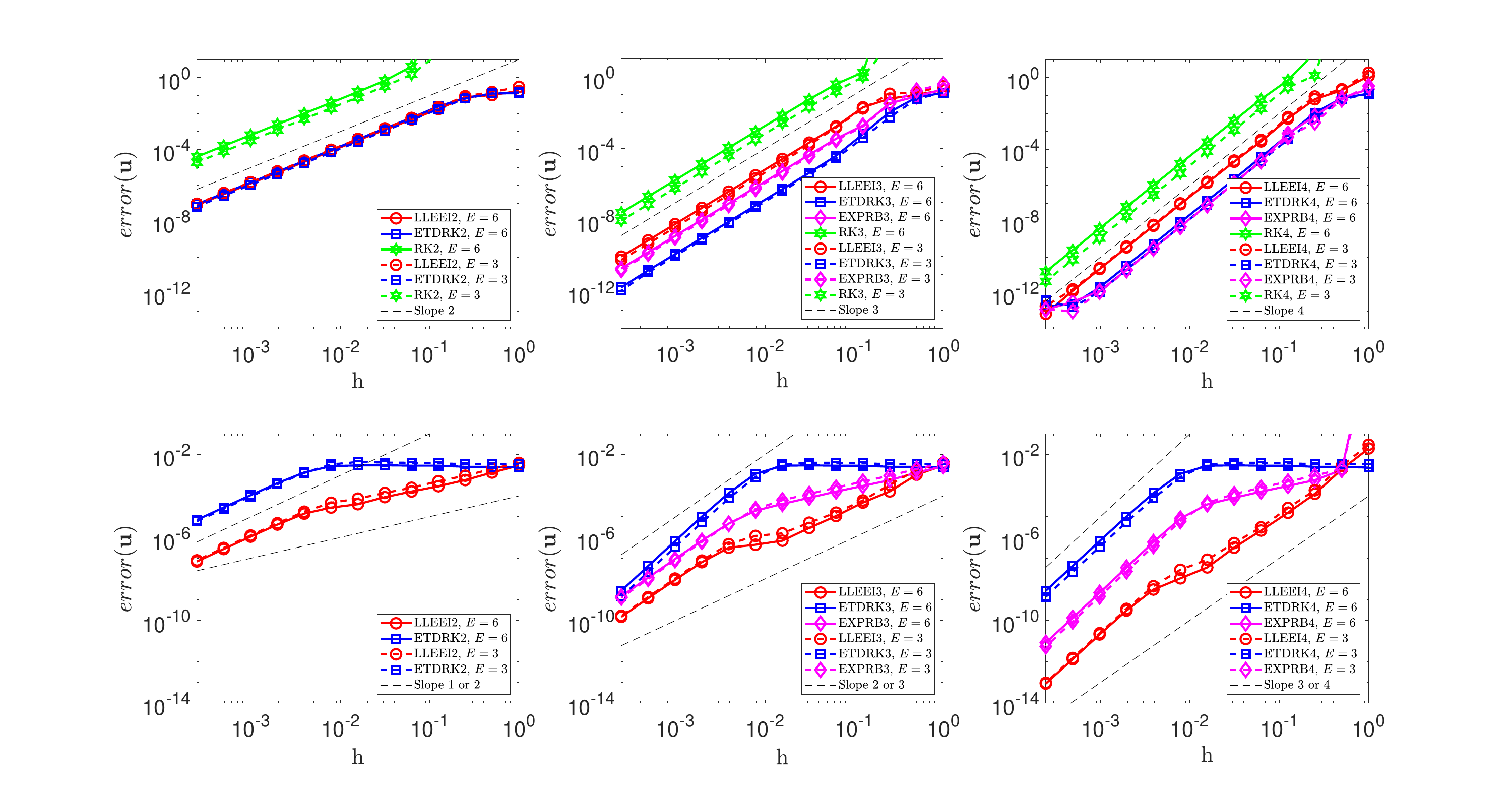}
  \caption{$error(\mathbf{u})$ versus step size $h$. The top row corresponds to the weakly oscillatory case ($\varepsilon=1/2^2$), while the bottom row represents the highly oscillatory regime ($\varepsilon=1/2^8$). The left, middle, and right columns correspond to second-, third-, and fourth-order methods, respectively.}\label{fig2-1}
\end{figure}

Fig. \ref{fig2-2} shows how the errors vary with $\varepsilon$ for two step sizes, $h=1/2^3$ and $h=1/2^8$.  Under small step sizes (bottom row), only the LLEEI methods preserve uniform accuracy in $\varepsilon$; the errors of all other tested schemes grow as $\varepsilon$ decreases. For large step sizes (top row), the errors of the LLEEI methods also consistently decrease with $\varepsilon$. Together, the bottom row of Fig.~\ref{fig2-1} and Fig. \ref{fig2-2} demonstrate that the uniform-accuracy property of the LLEEI schemes leads to superior performance over conventional exponential integrators in highly oscillatory regimes across all step sizes tested, with the advantage becoming more pronounced at higher orders.

\begin{figure}[htbp]
  \centering
  \includegraphics[width=\textwidth]{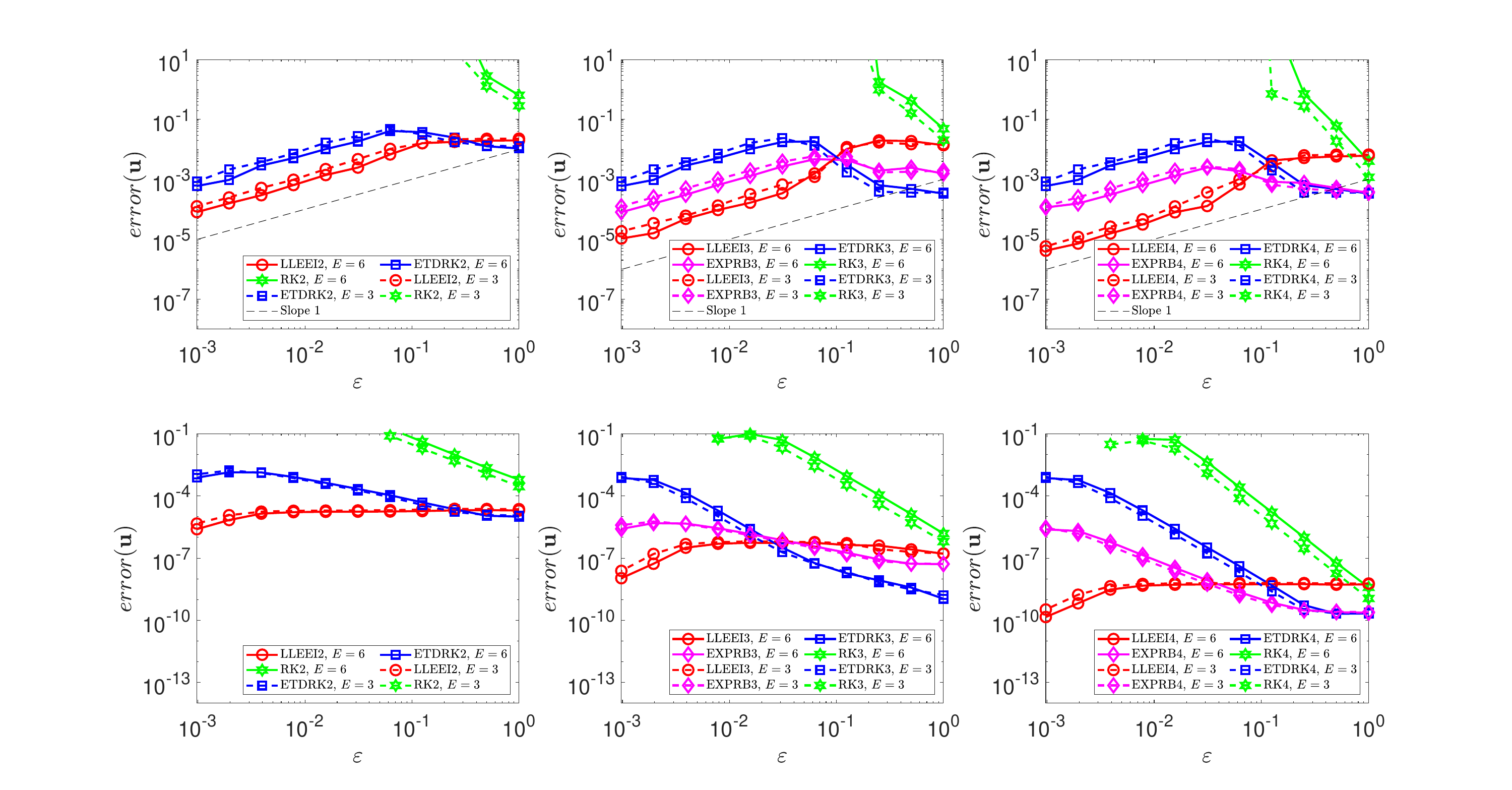}
  \caption{$error(\mathbf{u})$ versus $\varepsilon$. The top row corresponds to the large step size case ($h=1/2^3$), while the bottom row represents the small step size case ($h=1/2^8$). The left, middle, and right columns correspond to second-, third-, and fourth-order methods, respectively.}\label{fig2-2}
\end{figure}

\section{Conclusion}\label{Sec6}
This paper develops a rigorous theoretical framework for local linear extension EIs applied to highly oscillatory ODEs {with bounded oscillatory energy.} We first present a systematic review of the construction of local linear extension variables and their associated systems, followed by the derivation of the new class of EIs. Through introducing tensor product structures and performing precise algebraic operations, we demonstrate that the linear part of the high-dimensional extension system preserves essential spectral properties, including eigenvalue characteristics and diagonalizability. Some specific invariant subspaces of the differential operator and their matrix representations are analysed. These fundamental algebraic properties validate the effectiveness of the dimension-raising approach and further provide the theoretical basis for proving the uniform accuracy of the numerical schemes. The analysis establishes convergence results across different step size regimes. When solving highly oscillatory systems using large time steps $h>c_0\varepsilon$, some key elements play pivotal roles: the adiabatic transformation method, the periodic decomposition technique, and the superior algebraic properties of the local linear extension system. These allow separate analysis of oscillatory and non-oscillatory components. The numerical approach naturally yields solutions with improved uniform accuracy with respect to $\varepsilon$ for second-order non-autonomous equations. Numerical experiments confirm the optimality of these theoretical results.

\section*{Acknowledgement}
This work was supported by the National Key R\&D Program of China ( 2024YFA1012600 ) and the NSFC grant 12171237.

\begin{appendices}
\section{Algorithm for assembling high-dimensional matrices}\label{apdx}
We define a bijection $\tau_k:\tilde{\mathcal{I}}_{d+1}^{[k]}\to \{1,\cdots,D^{[k]}\}$ on the quotient set, which uniquely assigns each multi-index to a specific row or column of the matrix. To maintain consistency with the definition of local linear extension variables in Definition \ref{LLEV}, we set $\tau_k(\emptyset)=1$ and $\tau_k([j])=j+1$ for $j=1,\cdots,d+1$. The following algorithm outlines the construction of the matrices $A_1^{[k]}(\hat{\mathbf{x}})$ and $A_0^{[k]}(\hat{\mathbf{x}})$:
\begin{algorithm}
\caption{Construction of $A_1^{[k]}(\hat{\mathbf{x}})$ and $A_0^{[k]}(\hat{\mathbf{x}})$}\label{Alg01}
\begin{algorithmic}
\Require $\hat{\mathbf{x}}\in\mathbb{C}^{d+1}$, positive integer $k$, matrix $A_1$, function $\mathbf{f}(\mathbf{x})$, parameter $\varepsilon$, and sorting mapping $\tau_k$.
\Ensure matrices $A_1^{[k]}(\hat{\mathbf{x}})$ and $A_0^{[k]}(\hat{\mathbf{x}})$.
\State $A_1^{[k]}(\hat{\mathbf{x}})=A_0^{[k]}(\hat{\mathbf{x}})=\mathbf{0}_{D^{[k]}\times D^{[k]}}$.
\For{$i=2$ to $D^{[k]}$}
  \State $\overline{\alpha}=\tau_k^{-1}(i)$.
  \For{$l=1$ to $d+1$}
    \State $\alpha'=\chi(\overline{\alpha};l),i'=\tau_k(\overline{\alpha}')$
    \For{$m=1$ to $d+1$}
      \State $\left(A_1^{[k]}(\hat{\mathbf{x}})\right)_{i,i'} \leftarrow \left(A_1^{[k]}(\hat{\mathbf{x}})\right)_{i,i'}+(A_1)_{\alpha_lm}\hat{x}_m$.
      \State $\alpha''= \alpha'\oplus\{m\},i''=\tau(\overline{\alpha}'')$.
      \State $\left(A_1^{[k]}(\hat{\mathbf{x}})\right)_{i,i''} \leftarrow \left(A_1^{[k]}(\hat{\mathbf{x}})\right)_{i,i''}+(A_1)_{\alpha_lm}$.
    \EndFor
    \For{$m=0$ to $k-|\alpha|+1$}
      \For{$\overline{\beta}\in\tilde{\mathcal{I}}_{d+1}^{[[m]]}$}
        \State $\alpha''= \alpha' \oplus \overline{\beta},i''=\tau_k(\overline{\alpha}'')$,
        \State $\left(A_0^ {[k]}(\hat{\mathbf{x}})\right)_{i,i''} \leftarrow \left(A_0^ {[k]}(\hat{\mathbf{x}})\right)_{i,i''} + \frac{1}{\gamma(\overline{\beta})} \frac{\partial^{\overline{\beta}}\mathbf{f}(\hat{\mathbf{x}})}{\partial \mathbf{x}^{\overline{\beta}}}$.
      \EndFor
    \EndFor
  \EndFor
\EndFor
\end{algorithmic}
\end{algorithm}
\end{appendices}

\bibliography{ref}

\end{document}